\documentclass[oneside, 12pt]{amsart}
\usepackage{amsmath}
\usepackage{amssymb}
\usepackage{amsbsy}
\usepackage{color}
\usepackage[dvipsnames,x11names,svgnames]{xcolor}
\usepackage[hidelinks]{hyperref}
\usepackage{amsthm}
\usepackage{amscd}
\usepackage[margin=1in]{geometry}
\usepackage{mathrsfs}
\IfFileExists{dsfont.sty}{\usepackage{dsfont}}{}
\usepackage{mathtools}
\usepackage{enumitem}
\usepackage{tikz-cd}
\usepackage{cleveref}
\usepackage{graphicx}
\usepackage{lscape}

\newtheorem{thm}{Theorem}

\newtheorem{prop}[thm]{Proposition}

\newtheorem{lem}[thm]{Lemma}
\newtheorem{cor}[thm]{Corollary}
\crefname{thm}{theorem}{theorems}
\Crefname{thm}{Theorem}{Theorems}
\crefname{prop}{proposition}{propositions}
\Crefname{prop}{Proposition}{Propositions}
\crefname{lem}{lemma}{lemmas}
\Crefname{lem}{Lemma}{Lemmas}
\crefname{cor}{corollary}{corollaries}
\Crefname{cor}{Corollary}{Corollaries}
\theoremstyle{definition}
\newtheorem*{rem}{Remark}

\newcommand{\dt}{\, \textup{d}t}

\newcommand{\be}{\begin{equation*}}
	\newcommand{\ee}{\end{equation*} }

\newcommand{\ba}{\begin{align*}}
	\newcommand{\ea}{\end{align*}}

\newcommand{\ben}{\begin{equation}}
	\newcommand{\een}{\end{equation} }

\newcommand{\bs}{\begin{split}}
	\newcommand{\es}{\end{split}}

\newcommand{\bmu}{\begin{multline*}}
	\newcommand{\emu}{\end{multline*}}

\newcommand{\bmun}{\begin{multline}}
	\newcommand{\emun}{\end{multline}}

\author[M.~V.~Hagen]{Markus Valås Hagen}
\address{Department of Mathematical Sciences, Norwegian University of Science and Technology (NTNU), 7491 Trondheim, Norway} 
\email{markus.v.hagen@ntnu.no}

\thanks{Research supported in part by Grant 334466 of the Research Council of Norway.}

\subjclass{11M06, 11M35}

\begin{document}
	
	\title{Sharp unconditional moment bounds for products of $L$-functions}
	
	\begin{abstract}
		We investigate the mixed moments of the products of irreducible $\textup{GL}(1)$ and $\textup{GL}(2)$ $L$-functions on the critical line, and establish sharp moment bounds for these quantities. In particular this establishes the order of these mixed moments in certain ranges. This determines unconditionally, for the first time, the order of magnitude of the first moment of cubic Dedekind zeta functions with negative discriminant.  The proof of the lower bound introduces a new method for computing mixed moments, that is a mix of the classical method of Heath--Brown with the modern machinery of Harper, Heap, Radziwi\l\l\, and Soundararajan. 
	\end{abstract}
	\maketitle
	
	\section{Introduction}
	In analytic number theory, the study of moments and its ramifications is one of the most central ones. In general it refers to averaging some object $I(f)$ over some family $\mathcal{F}$, often in the following form $\frac{1}{|\mathcal{F}|}\sum_{f \in \mathcal{F}} |I(f)|^{2k}$. One of the initial main motivations for studying these objects is that they provide good distributional information for the $L$-functions in question. In this paper we are interested in this study in the setting of mixed moments of $L$-functions on the critical line.
	
	More specifically, let $L_1,\dots,L_r$ be distinct $L$-functions with degrees $d_1,\dots,d_r \in \{1,2\}$. If $d_i = 1$, we allow $L_i$ to be either the Riemann zeta function $\zeta(s)$ or a Dirichlet $L$-function $L(s,\chi)$ associated to a primitive Dirichlet character. If $d_i=2$, $L_i$ can be any $L$-function $L(s,f)$ where $f$ is a primitive holomorphic Hecke cusp form of integral weight $\kappa \geq 1$ and unitary nebentypus. The mixed moments we are interested in take the form
	\begin{equation}\label{mixedMomentDefinition}
		I_{\boldsymbol k}(T)=I_{k_1,\ldots,k_r}(T)
		:=\int_T^{2T}|L_1(\tfrac12+it)|^{2k_1}\cdots |L_r(\tfrac12+it)|^{2k_r}\dt,
	\end{equation}
	where $\boldsymbol{k} = (k_1,\dots,k_r)$ is fixed and $k_1,\ldots,k_r>0$.

	To elaborate more on the statistical interpretation of moments mentioned above, it is illustrative to look at how moments shed light on the independence between $L$-functions. Indeed they can be used to study the large deviation regime of the distributional function $$\Phi_T(V_1,\dots,V_r) = \frac{1}{T}\textup{meas}\left( t \in [1,T] : \frac{\log|L_j(\tfrac12+it)|}{\sqrt{\frac12 \log \log T}}\geq V_j, 1\leq j \leq r\right)$$ in the large deviations regime $V_j \asymp \sqrt{\log\log T}$, like in \cite[Theorem 1.3]{Hagen1}. 
	
	If one is willing to assume the Grand Riemann Hypothesis (GRH) and the Generalized Ramanujan--Petersson conjecture, the order of $I_{\boldsymbol{k}}(T)$ is known for any $\boldsymbol{k}\in \mathbb{R}^r_{>0}$: $$I_{\boldsymbol{k}}(T)\asymp T(\log T)^{k_1^2+\dots+k_r^2}$$ where each of $L_1,\dots,L_r$ is allowed to be any irreducible cuspidal automorphic $L$-function with unitary character. This result is due to the author \cite{Hagen1}, and sharpens the earlier result $I_{\boldsymbol{k}}(T)\ll T(\log T)^{k_1^2+\dots+k_r^2+\varepsilon}$ of Milinovich--Turnage-Butterbaugh \cite{MilinoButter1}. 
	
	In the unconditional case however, the range in which the results above are valid is more limited. For asymptotics, there are only a handful of cases. Indeed, if the exponents are equal, the only known case in which the asymptotic is known is Motohashi's asymptotic for a second moment of a quadratic field \cite{Motohashi1}. This can also be generalized slightly to a product of two irreducible Dirichlet $L$-functions if one follows the argument of Heap's PhD Thesis \cite{Heap2}. As far as the author can see, no upper bound up to order has been recorded beyond these two examples.
	
	For lower bounds, the story is a bit more positive. For positive
	rational $k_j$, sharp lower bounds follow from the general theorem
	of Akbary--Fodden \cite{AkbaryFodden1}, which extends Heath--Brown's
	method. The rationality restriction comes from this method:
	extending the argument directly to arbitrary real exponents
	encounters difficulties in choosing analytic branches in the
	presence of zeros. Assuming GRH removes this obstruction,
	but we are seeking unconditional bounds.
	
	For common real exponents $k>1$, Sono \cite{Sono1} obtained sharp
	lower bounds for Dedekind zeta functions using the method of
	Radziwi\l\l--Soundararajan \cite{RS1}. For the Riemann zeta
	function, Heap--Soundararajan \cite{HeapSound1} obtained
	unconditional sharp lower bounds for every fixed real $k>0$. The next two theorems extend these results in the setting of mixed moments. The purpose of this paper is to prove them.
	
	\begin{thm}\label{mainThm1}
		For any fixed $k_1,\ldots,k_r>0$ such that $d_1k_1+\dots+d_rk_r \leq 2$,
		\[
		I_{k_1,\dots,k_r}(T) \ll
		T(\log T)^{k_1^2+\dots+k_r^2}.
		\]
	\end{thm}
	
	\begin{thm}\label{mainThm2}
		For any fixed $k_1,\dots,k_r>0$
		\[
		I_{k_1,\dots,k_r}(T)\gg
		T(\log T)^{k_1^2+\dots+k_r^2}.
		\]
	\end{thm}
	
	In particular, the order $$I_{k_1,\dots,k_r}(T)\asymp T(\log T)^{k_1^2+\dots+k_r^2}$$ is established for $d_1k_1+\dots+d_rk_r \leq 2$.
	
	\begin{rem}
		While this manuscript was in its final preparation, a noteworthy result was announced by Gun, Kumar and Thakur \cite{GunKumarThakur} on arxiv. They partially recover, and improve our lower bound theorem. In particular they establish sharp lower bounds for the moment $$\int_T^{2T} |L(\tfrac12+it,\pi)|^{2k} \dt$$ for any $k>0$, where $\pi$ is an irreducible unitary cuspidal
		automorphic representation of degree at most four. For higher degree automorphic representations it assumes Hypothesis H of Rudnick--Sarnak.  Their theorem also covers mixed moments whose exponents have rational
		ratios, extending the rational-exponent results of
		Akbary--Fodden \cite{AkbaryFodden1}, whereas
		Theorem~\ref{mainThm2} allows arbitrary positive real exponents.
	\end{rem}

	As a nice consequence of our main theorems we are able to obtain the order of the moment of some Dedekind zeta functions. The next corollary presents this.
	
	\begin{cor}\label{dedekindZetaCorollary}
		Let $K$ be a number field with degree $n=[K:\mathbb{Q}]$. Let $M$ be the Galois closure of $K$ and put $G=\textup{Gal}(M/\mathbb{Q})$ and $H=\textup{Gal}(M/K)$. Assume that every  irreducible constituent of the representation $\textup{Ind}_{H}^G \mathbf{1}_H$ is either one-dimensional or has dimension two and is such that its Artin $L$-function $L(s,\rho)$ equals $L(s,f_\rho)$ for some primitive holomorphic weight-one form. Write $$\textup{Ind}_{H}^G \mathbf{1}_H = \bigoplus_{j=1}^{\eta} \rho_j^{m_j}, \qquad \mathfrak{P}_K = \sum_{j=1}^{\eta} m_j^2.$$  Then
		\[
		\int_T^{2T}|\zeta_K(\tfrac12+it)|^{2k}\,dt
		\asymp T(\log T)^{\mathfrak P_Kk^2}
		\]
		whenever
		$0<k\leq2/[K:\mathbb Q]$.
		The lower bound is unconditional for every $k>0$.
	\end{cor}
	
	\begin{proof}
		The proof follows from considering the Artin $L$-function attached to $\textup{Ind}_H^G \mathbf{1}_H$, and some basic algebraic number theory and representation theory. The details are left to the reader.
	\end{proof}
	
	As an example of the corollary above we record the order of the Dedekind zeta function $\zeta_K$ of any cubic field with negative discriminant, $$\int_T^{2T} |\zeta_K(\tfrac12+it)|^{2k}\dt \asymp T(\log T)^{2k^2},$$ for any $0<k\leq 2/3$.  Two remarks are in order.
	
	\begin{rem}
		The lower bound in Corollary \ref{dedekindZetaCorollary} also appears in the very recent preprint by Gun--Kumar--Thakur mentioned above \cite{GunKumarThakur}. In fact they prove it for any Dedekind zeta function $\zeta_K$. 
	\end{rem}

	\begin{rem}
		The corollary also covers cyclic cubic fields. The non-Galois
		cubic case with positive discriminant involves a dihedral Maass
		form in the $L$-function factorization and thus lies outside the scope of this paper.
	\end{rem}

	There are also other interesting consequences of knowing the order of these mixed moments: for example sharp moment bounds for Hurwitz zeta functions of rational parameter (see \cite[Corollary 1.6]{Hagen1}) or variance of arithmetic functions in short intervals (see \cite[Theorem 1.4]{MilinoButter1}). We refer the reader to the mentioned articles and the references therein for more details on these aspects.

	The rest of the introduction is devoted to discussing our proof strategy, possible extensions, and finally a comment on the structure of the paper.
	
	\subsubsection*{The upper bounds strategy.} The proof strategy of Theorem \ref{mainThm1} is a direct adaptation of the method of Harper \cite{Harper1}, Radziwi\l\l--Soundararajan \cite{RS2} and Heap--Radziwi\l\l--Soundararajan \cite{HeapRadSound1}. Expositions of these methods can be found in many places in the literature (see e.g. the proof strategy section in the author's paper \cite{Hagen1} or Heap's paper \cite{Heap3}), and in the interest of keeping the paper as short as possible, we have to decided to not dwell to much on it here. The main difference from the implementation of this method in \cite{HeapRadSound1} is the fact that we are considering several $L$-functions at once here, and so the combinatorics become more cumbersome. 
	
	\subsubsection*{The lower bounds strategy} The proof strategy for Theorem \ref{mainThm2} introduces a novel unconditional method for computing mixed moment bounds. To motivate our method it is illustrative to look at why the Heap--Soundararajan method when adapted naively to our case runs into difficulty. The essential underlying reason is that when several $L$-functions are involved
	and the exponents are small, the argument of Heap--Soundararajan can stop being
	interpolating. Consider three $L$-functions
	and suppose that suitable twisted second-moment estimates
	are available.
	Writing $L_j=L_j(\tfrac12+it)$ and suppressing the Dirichlet polynomials,
	one would try to apply Hölder's inequality in the form
	\[
	\begin{aligned}
		\left|\int_T^{2T}L_1L_2L_3\dt\right|
		&\leq
		\left(\int_T^{2T}
		|L_1|^{2k_1}|L_2|^{2k_2}|L_3|^{2k_3}\dt\right)^\alpha\\
		&\quad{}\times
		\left(\int_T^{2T}|L_1|^2\dt\right)^\beta
		\left(\int_T^{2T}|L_2|^2\dt\right)^\gamma
		\left(\int_T^{2T}|L_3|^2\dt\right)^\delta.
	\end{aligned}
	\]
	There is usually an additional factor involving only Dirichlet
	polynomials. This uses a non-negative Hölder weight, so omitting it
	only makes the conditions easier to satisfy. Even without that factor,
	Hölder requires non-negative weights satisfying
	\[
	\begin{gathered}
		\alpha+\beta+\gamma+\delta=1,\\
		2k_1\alpha+2\beta
		=2k_2\alpha+2\gamma
		=2k_3\alpha+2\delta=1.
	\end{gathered}
	\]
	The last equality comes from balancing out both sides of the inequality, i.e. trying to make it as sharp as possible. Writing $K=k_1+k_2+k_3$, these conditions give
	\[
	\alpha=\frac{1}{2(K-1)},
	\]
	which is negative when $K<1$. The obstruction is that the proposed
	argument is no longer interpolating. Indeed, after taking absolute
	values, identifying $L_1=L_2=L_3$ would require controlling a third
	absolute moment using moments of orders $2K$ and $2$. For small $k_j$,
	both available orders lie below the order on the left.
	
	To restore interpolation, we take inspiration from an unpublished idea
	of Radziwi\l\l, that was communicated to the author by Heap. It is based on Heath--Brown's method \cite{Heath-Brown1}.
	Instead of the first moment above on the left-hand side of the inequality above, we consider a suitably weighted
	version of
	\[
	\int_T^{2T}|L_1L_2L_3|^{1/n}\dt
	\]
	for a sufficiently large fixed integer $n$. Lowering the powers on the
	left leaves room in H\"older's inequality for the additional factor
	involving only Dirichlet polynomials. The required weights can then be
	chosen positive, restoring the interpolation between the target mixed
	moment, the available twisted moments, and the pure Dirichlet-polynomial
	moment. The absolute values allow us to obtain the required lower bound
	for the weighted fractional moment by adapting Heath--Brown's method. On the right-hand side of the inequality we adapt the method of Heap and Soundararajan for the computational aspects. Similar to the upper bound, the combinatorics becomes more cumbersome in this setting compared to the single $L$-function case.
	
	The lowering-mechanism via Heath-Brown also appears in Gun--Kumar--Thakur \cite{GunKumarThakur}, but the motivation is a bit different. In their method it allows the authors to completely circumvent the need for twisted moment formulae. This is essentially what allows them to make a result that is valid for any irreducible unitary cuspidal automorphic representation under the hypotheses stated above. 
	
	A related approach to ours also appears in Sourmelidis \cite{Sourmelidis1},
	who lowers the common exponent so that Hölder's inequality
	can be applied using individual twisted second moments.
	He combines these with the
	resonance method to prove unconditional simultaneous large values
	of arbitrarily many distinct primitive $\mathrm{GL}(1)$ and
	$\mathrm{GL}(2)$ $L$-functions.
	
	It is an interesting question whether the method of Gun--Kumar--Thakur \cite{GunKumarThakur} can be adapted to the mixed-moment case with arbitrary
	positive real exponents. In particular, it would be interesting to see if we could make our lower bound method independent of twisted moment formulae. This would allow us to extend Theorem \ref{mainThm2} to degree three and four, and higher degree on assuming Hypothesis H.
	
	Another interesting direction to look at further is if the methods presented here, perhaps in conjunction with \cite{GunKumarThakur}, can be adapted to remove the GRH-assumption in the recent preprint of Durkan, Karak and Mahatab \cite{DurkanKarakMahatab} where they exhibit sharp lower bounds for mixed moments of Dedekind zeta functions.
	
	\subsubsection*{Structure of the paper}
	We have tried to make the paper as self-contained as possible. For example the reader might find an appendix on Rankin--Selberg convolution, which contains the information that we need in the main body of the paper, but which any expert can safely skip. The structure of the paper is as follows. Section~\ref{holderReductionsSection} isolates the Hölder reductions that prove the two main theorems.  Section~\ref{twistedMomentsSection} records the twisted-moment formulae and the Dirichlet-polynomial estimates needed later.  Section~\ref{momentComputationsSection} constructs the factorwise prime blocks and establishes the analytic mean-value inputs used by the reductions.  The appendices contain the classical twisted-moment identities, combinatorial identities and Rankin--Selberg prime sums.

	\section*{Acknowledgements}
	I am deeply grateful to my supervisors Winston Heap and Kristian Seip for {discussions} and encouragement. This has been work in progress since late 2024. I am grateful to Junxian Li for comments on earlier versions of this paper which led to a significantly more general result. I also want to thank Louis-Pierre Arguin and the number theory group at the University of Oxford, for letting me stay there May--June 2025. The excellent working conditions provided there inspired some parts of this paper. I also want to thank Maksym Radziwi\l\l\, for letting me use his unpublished idea.
	
	The author also acknowledges the use of OpenAI's ChatGPT during the development of portions of the combinatorial parts, and for assisting in writing up technical details of the paper.

	{
		\section{Hölder reductions and proofs of the main theorems}\label{holderReductionsSection}
		
		For both the upper and lower bounds, we begin by applying
		Hölder's inequality to reduce the proof to certain moment computations.
		We call this step a \emph{reduction}.  The three reductions in this
		paper happen in \eqref{upperPointwiseBalancedIdentity},
		\eqref{caseAHolder}, and \eqref{caseBHolder}.
		We first introduce the prime-block construction and the Taylor
		approximation used in both proofs.
		
		\subsection{Prime-block notation and Taylor approximation}
		
		Write \(\boldsymbol L=(L_1,\ldots,L_r)\) for the fixed tuple of $L$-functions from the
		introduction and
		\[
		L_a(s)=\sum_{n\geq1}\frac{\lambda_a(n)}{n^s}.
		\]
		Thus \(\lambda_a(p)\) is respectively \(1\), \(\chi_a(p)\), or
		\(\lambda_{f_a}(p)\), according as \(L_a\) is zeta, a primitive
		Dirichlet \(L\)-function, or a primitive holomorphic cusp-form
		\(L\)-function.  We omit from every prime sum the finite set of primes
		ramified in at least one member of \(\boldsymbol L\).\footnote{
			Here a prime is called ramified if it divides the modulus
			of a Dirichlet character or the level of a cusp form
			occurring in $L_1,\ldots,L_r$.}  Put
		\[
		d_a=\deg L_a\in\{1,2\},\qquad
		\kappa(\boldsymbol x)=\sum_{a=1}^r x_a^2,
		\]
		and denote by \(\boldsymbol e_i\) the \(i\)th coordinate vector and by
		\(\boldsymbol 1=(1,\ldots,1)\).
		
		By a \emph{Taylor parameter} we mean a real number $\alpha$
		occurring as the second argument of
		$\mathcal N_{a,j}(s;\alpha)$.\footnote{As we will see shortly,
			$\mathcal N_{a,j}(s;\alpha)$ is a truncation of the Taylor series
			for $\exp(\alpha\mathcal P_{a,j}(s)/2)$, which explains the name.}
		For each reduction, $\mathscr A$ denotes the finite set of Taylor
		parameters used.  We specify this set before choosing the constants
		and forming the prime blocks.
		
		The remaining constants are chosen in the following order.
		First fix $D_0\geq1$, depending only on $\boldsymbol L$, such that,
		for the blocks $(T_{j-1},T_j]$ defined a few lines below, uniformly in $j$,
		\begin{equation}\label{unweightedVarianceComparison}
			\sum_{T_{j-1}<p\leq T_j}\frac1p
			\leq D_0\left(1+\max_a
			\sum_{T_{j-1}<p\leq T_j}\frac{|\lambda_a(p)|^2}{p}\right).
		\end{equation}
		The existence of such a $D_0$ follows from Mertens' theorem
		and \Cref{rankinSelbergPrimeSumsProposition}.
		
		Next choose the constant $C_{\rm cut}$ with
		\begin{equation}\label{baseBadCutoffChoice}
			C_{\rm cut}\geq4 D_0r e^{16}.
		\end{equation}
		With $C_{\rm cut}$ fixed, suppose for now that
		\begin{equation}\label{YDominatesParameters}
			\widehat Y\geq1+\max_{\alpha\in\mathscr A}
			\bigl(|\alpha|+\alpha^2\bigr).
		\end{equation}
		The final choice of $\widehat Y$ will be made later and may
		also depend on $C_{\rm cut}$, so the order of these choices matters.
		Once $\widehat Y$ is fixed, choose a sufficiently large constant
		$\Lambda\geq10^7rC_{\rm cut}\widehat Y^2$.
		Finally, take $T$ sufficiently large. 
		For each reduction, the constants are chosen once, after its finite
		parameter sets have been specified.
		
		Let \(\log_j\) denote the \(j\)-fold iterated logarithm.  Set
		\[
		T_1=\exp(2\widehat Y^2),\qquad
		T_j=\exp\left(\frac{\log T}{\widehat Y(\log_jT)^2}\right)
		\quad(j\geq2).
		\]
		Let \(\mathscr J\)
		be the largest integer \(j\geq2\) for which
		\(\log_jT\geq\Lambda\).  If $x_j=\log_jT$, then
		\begin{equation}\label{terminalScaleBounds}
			\Lambda\leq x_{\mathscr J}<e^\Lambda.
		\end{equation}
		For
		\(2\leq j\leq\mathscr J\), define
		\[
		\mathcal P_{a,j}(s)=\sum_{T_{j-1}<p\leq T_j}
		\frac{\lambda_a(p)}{p^s},\qquad
		P_j=1+\max_{1\leq a\leq r}
		\sum_{T_{j-1}<p\leq T_j}\frac{|\lambda_a(p)|^2}{p}.
		\]
		Put
		\begin{equation}\label{factorwiseCutoffs}
			K_j=\lceil10C_{\rm cut}\widehat YP_j\rceil,
			\qquad \ell_j=\lfloor C_{\rm cut}P_j\rfloor,
			\qquad E_K(z)=\sum_{m=0}^K\frac{z^m}{m!},
		\end{equation}
		and, separately for every \(L_a\) and every prime block $(T_{j-1},T_j]$, define
		\begin{equation}\label{factorwiseNDefinition}
			\mathcal N_{a,j}(s;\alpha)=
			E_{K_j}\left(\frac{\alpha}{2}\mathcal P_{a,j}(s)\right).
		\end{equation}
		For a fixed $t\in[T,2T]$, we call the $j$th block
		\emph{good} for $L_a$ if
		$|\mathcal P_{a,j}(\tfrac12+it)|\leq C_{\rm cut}P_j$,
		and \emph{bad} otherwise.
		
		For \(\boldsymbol\nu=(\nu_1,\ldots,\nu_r)\in
		\{2,\ldots,\mathscr J+1\}^r\), let
		\(\mathscr B(\boldsymbol\nu)\) be the set of \(t\in[T,2T]\) such that,
		for every \(a\),
		\[
		|\mathcal P_{a,j}(\tfrac12+it)|\leq C_{\rm cut}P_j
		\quad(2\leq j<\nu_a),
		\]
		and, when \(\nu_a\leq\mathscr J\),
		\[
		|\mathcal P_{a,\nu_a}(\tfrac12+it)|>C_{\rm cut}P_{\nu_a}.
		\]
		The sets \(\mathscr B(\boldsymbol\nu)\) form a disjoint partition of
		\([T,2T]\); the value \(\nu_a=\mathscr J+1\) means that every block
		belonging to \(L_a\) is good. Otherwise, $\nu_a$ is the first bad block
		for $L_a$; no condition is imposed on blocks with $j>\nu_a$. For
		\(\boldsymbol\alpha=(\alpha_1,\ldots,\alpha_r)\in\mathscr A^r\), define
		\begin{align}
			\mathcal N_{\boldsymbol\nu}(s;\boldsymbol\alpha)
			&=\prod_{a=1}^r\prod_{2\leq j<\nu_a}
			\mathcal N_{a,j}(s;\alpha_a),\label{truncatedVectorN}\\
			\mathcal Q_{\boldsymbol\nu}(s)
			&=\prod_{\substack{1\leq a\leq r\\\nu_a\leq\mathscr J}}
			\left(\frac{\mathcal P_{a,\nu_a}(s)}{C_{\rm cut}P_{\nu_a}}
			\right)^{\ell_{\nu_a}}.\label{badBlockPolynomial}
		\end{align}
		Then
		\[
		\mathbf1_{\mathscr B(\boldsymbol\nu)}(t)
		\leq|\mathcal Q_{\boldsymbol\nu}(\tfrac12+it)|^2.
		\]
		We abbreviate
		\[
		\mathscr G=\mathscr B(\mathscr J+1,\ldots,\mathscr J+1),
		\qquad
		\mathcal N_{\mathscr J+1}(s;\boldsymbol\alpha)
		=\mathcal N_{(\mathscr J+1,\ldots,\mathscr J+1)}
		(s;\boldsymbol\alpha).
		\]
		The endpoint choice and the Rankin--Selberg prime sums imply, for every
		fixed $c>0$,
		\begin{equation}\label{cutoffSummability}
			\sum_{j=2}^{\mathscr J}e^{-cK_j}
			\ll_{c,\boldsymbol L,r,\mathscr A,C_{\rm cut},\widehat Y}1.
		\end{equation}
		\begin{lem}\label{factorwiseExponentialLemma}
			Let $\mathscr A$ and $\widehat Y$ satisfy
			\eqref{YDominatesParameters}.
			If $|\mathcal P_{a,j}(s)|\leq C_{\rm cut}P_j$, then, uniformly for
			$\alpha\in\mathscr A$,
			\begin{equation}\label{factorwiseExponentialApproximation}
				|\mathcal N_{a,j}(s;\alpha)|^2
				=\exp\bigl(\alpha\Re\mathcal P_{a,j}(s)\bigr)
				\left(1+O_{\mathscr A}(e^{-K_j/5})\right).
			\end{equation}
			Consequently, if $c>0$ is fixed and $\alpha,c\alpha\in\mathscr A$, then
			\begin{equation}\label{moveRealPowerInside}
				|\mathcal N_{a,j}(s;\alpha)|^{2c}
				=|\mathcal N_{a,j}(s;c\alpha)|^2
				\left(1+O_{\mathscr A,c}(e^{-K_j/6})\right).
			\end{equation}
		\end{lem}
		
		The proof of \Cref{factorwiseExponentialLemma} is given in
		Section~\ref{momentComputationsSection}.
		
		\subsection{The upper-bound reduction}
		
		For the upper-bound reduction assume
		\begin{equation}\label{upperAdmissibleRange}
			\sum_{a=1}^r d_ak_a\leq2
		\end{equation}
		and put
		\[
		\vartheta_a=\frac{d_ak_a}{2},
		\qquad
		\vartheta_0=1-\sum_{a=1}^r\vartheta_a\geq0.
		\]
		These are the Hölder weights for later - the last one is omitted if $\vartheta_0=0$.
		For $1\leq i\leq r$ put
		\[
		\boldsymbol\alpha^{(i)}=2\boldsymbol k-\frac4{d_i}\boldsymbol e_i,
		\qquad
		\boldsymbol\alpha^{(0)}=2\boldsymbol k.
		\]
		Here the Taylor parameter set is
		\[
		\mathscr A=\{2k_a:1\leq a\leq r\}
		\cup\{2k_a-(4/d_i)\mathbf1_{a=i}:
		1\leq a,i\leq r\}.
		\]
		Choose the constants in the prescribed order.  On
		$\mathscr B(\boldsymbol\nu)$, \Cref{factorwiseExponentialLemma},
		\eqref{cutoffSummability}, and
		$\mathbf1_{\mathscr B(\boldsymbol\nu)}\leq|\mathcal Q_{\boldsymbol\nu}|^2$
		give
		\begin{align}
			&\mathbf1_{\mathscr B(\boldsymbol\nu)}
			\prod_{a=1}^r|L_a(\tfrac12+it)|^{2k_a}\nonumber\\
			&\quad\ll
			\prod_{a=1}^r\left(
			|L_a(\tfrac12+it)|^{4/d_a}
			|\mathcal N_{\boldsymbol\nu}
			(\tfrac12+it;\boldsymbol\alpha^{(a)})
			\mathcal Q_{\boldsymbol\nu}(\tfrac12+it)|^2
			\right)^{\vartheta_a}\nonumber\\
			&\qquad\times\left(
			|\mathcal N_{\boldsymbol\nu}
			(\tfrac12+it;\boldsymbol\alpha^{(0)})
			\mathcal Q_{\boldsymbol\nu}(\tfrac12+it)|^2
			\right)^{\vartheta_0}.
			\label{upperPointwiseBalancedIdentity}
		\end{align}
		Here the final factor is omitted if $\vartheta_0=0$.
		The powers balance because
		$(4/d_i)\vartheta_i=2k_i$ and
		\[
		\sum_{i=1}^r\vartheta_i(2k_a-(4/d_i)\mathbf1_{i=a})
		+2k_a\vartheta_0=0.
		\]
		For each $\nu$, integrating \eqref{upperPointwiseBalancedIdentity}
		over $[T,2T]$ and applying H\"older's inequality with exponents
		$1/\vartheta_a$, for $0\le a\le r$ with $\vartheta_a>0$, bounds
		the integral over $B(\nu)$ in terms of the following moments.
		\begin{prop}\label{unifiedMeanValuesProposition}
			For any fixed finite Taylor parameter set $\mathscr A$ and its
			prime-block construction, let $\boldsymbol\alpha\in\mathscr A^r$ and
			$\boldsymbol\nu\in\{2,\ldots,\mathscr J+1\}^r$, and fix
			$\varepsilon>0$.  There are constants
			$c=c(\boldsymbol L,r,\mathscr A,C_{\rm cut},\widehat Y)>0$ and
			$C_{\rm tw}=C_{\rm tw}(\boldsymbol L,r,\mathscr A)>0$, independent of
			the terminal parameter $\Lambda$, such that
			\begin{align}
				&\int_T^{2T}
				|\mathcal N_{\boldsymbol\nu}(\tfrac12+it;\boldsymbol\alpha)
				\mathcal Q_{\boldsymbol\nu}(\tfrac12+it)|^2\,dt\nonumber\\
				&\qquad\ll T(\log T)^{\frac14\sum_{a=1}^r\alpha_a^2}
				\exp\left(-c\sum_{\nu_a\leq\mathscr J}K_{\nu_a}\right).
				\label{unifiedUntwistedMean}
			\end{align}
			For each $1\leq i\leq r$ one also has
			\begin{align}
				&\int_T^{2T}|L_i(\tfrac12+it)|^{4/d_i}
				|\mathcal N_{\boldsymbol\nu}(\tfrac12+it;\boldsymbol\alpha)
				\mathcal Q_{\boldsymbol\nu}(\tfrac12+it)|^2\,dt\nonumber\\
				&\qquad\ll_{\boldsymbol L,r,\mathscr A,\varepsilon}
				x_{\mathscr J}^{C_{\rm tw}}
				T(\log T)^{\frac14\{(4/d_i+\alpha_i)^2+
					\sum_{a\neq i}\alpha_a^2\}}
				\exp\left(-c\sum_{\nu_a\leq\mathscr J}K_{\nu_a}\right).
				\label{unifiedTwistedMean}
			\end{align}
		\end{prop}
		
		By \Cref{unifiedMeanValuesProposition}, each factor has logarithmic
		exponent $\kappa(\boldsymbol k)$. Since $x_{\mathscr J}<e^\Lambda$ by
		\eqref{terminalScaleBounds} and $\Lambda$ is fixed, the extra powers
		of $x_{\mathscr J}$ are bounded independently of $T$, giving
		\[
		\int_{\mathscr B(\boldsymbol\nu)}
		\prod_a|L_a(\tfrac12+it)|^{2k_a}dt
		\ll T(\log T)^{\kappa(\boldsymbol k)}
		\exp\left(-c\sum_{\nu_a\leq\mathscr J}K_{\nu_a}\right).
		\]
		Summing over the disjoint sets $\mathscr B(\boldsymbol\nu)$ proves
		\Cref{mainThm1}, since \eqref{cutoffSummability} gives
		\[
		\sum_{\boldsymbol\nu\in\{2,\ldots,\mathscr J+1\}^r}
		\exp\left(-c\sum_{\nu_a\leq\mathscr J}K_{\nu_a}\right)
		=\left(1+\sum_{j=2}^{\mathscr J}e^{-cK_j}\right)^r\ll1.
		\]
		
		The proof of \Cref{unifiedMeanValuesProposition} is given in
		Section~\ref{momentComputationsSection}.
		
		\subsection{The lower-bound reduction}
		
		Fix arbitrary positive $k_1,\ldots,k_r$.  We apply Hölder's
		inequality in the two cases below, using the real-power comparison
		of \Cref{factorwiseExponentialLemma}. 
		
		Let $\mathscr A_0$ contain the Taylor parameters used directly, and let
		$\mathscr E$ contain the pairs $(\alpha,c)$ for which we use
		\eqref{moveRealPowerInside}.  In each case we specify these finite sets
		and take
		\[
		\mathscr A=\mathscr A_0\cup\{\alpha,c\alpha:(\alpha,c)\in\mathscr E\}.
		\]
		For $\boldsymbol c>0$ and a positive integer $n$, the lower estimates
		below use the additional Taylor parameters
		\begin{equation}\label{oneOverNFiniteParameterSet}
			\mathscr A_0^{(1/n)}(\boldsymbol c,n)
			=\{2c_a,\,2c_a-1/n:1\leq a\leq r\}
			\cup\{2c_a-(4/d_i)\mathbf1_{a=i}:1\leq a,i\leq r\}.
		\end{equation}
		Put
		\[
		\rho(x)=\max(x,x^{-1}),\qquad
		R=\prod_{a=1}^r\rho(k_a),\qquad
		A_a=1-\frac{k_a}{R}.
		\]
		Since $R\geq k_a$, every $A_a\geq0$.
		\subsubsection*{Case A: $A_a>0$ for every $a$}
		
		Choose an integer
		\[
		n>\max\left\{
		\frac1{2\min_ak_a},
		\frac r2,
		\frac1{2R}+\frac14\sum_{a=1}^rd_aA_a
		\right\}
		\]
		and define $\delta>0$ by
		\begin{equation}\label{caseADeltaDefinition}
			\frac1\delta=n-\frac1{2R}-\frac14\sum_{a=1}^rd_aA_a.
		\end{equation}
		For $1\leq j,a\leq r$ set
		\[
		\mu_{j,a}=\frac{d_jA_j}{2}
		\left(2k_a-\frac4{d_j}\mathbf1_{a=j}\right),
		\qquad
		\mu_{0,a}=\frac{4k_a}{\delta}.
		\]
		The finite parameter sets are
		\begin{align*}
			\mathscr A_0^{(A)}={}&
			\mathscr A_0^{(1/n)}(\boldsymbol k,n)
			\cup\{\mu_{j,a}:0\leq j\leq r,\ 1\leq a\leq r\},\\
			\mathscr E^{(A)}={}&
			\left\{\left(\mu_{j,a},\frac{2}{d_jA_j}\right),
			\left(\mu_{j,a},\frac1{2n}\right):1\leq j,a\leq r\right\}\\
			&\cup\left\{\left(\mu_{0,a},\frac\delta2\right),
			\left(\mu_{0,a},\frac1{2n}\right):1\leq a\leq r\right\}.
		\end{align*}
		Use \eqref{residualMajorantParameters} with $\boldsymbol c=\boldsymbol k$.
		Choose $C_{\rm cut}$ by \eqref{baseBadCutoffChoice} and
		\eqref{residualCutoffChoice}, then $\widehat Y$ sufficiently large, satisfying
		\eqref{YDominatesParameters} and \eqref{oneOverNParameterDominance},
		and finally $\Lambda$ as in \Cref{modifiedHeathBrownOneOverN}.
		Define
		\begin{align*}
			\mathcal M_j(s)&=
			\prod_{a=1}^r\prod_{\ell=2}^{\mathscr J}
			\mathcal N_{a,\ell}(s;\mu_{j,a})\quad(1\leq j\leq r),\\
			\mathcal M_0(s)&=
			\prod_{a=1}^r\prod_{\ell=2}^{\mathscr J}
			\mathcal N_{a,\ell}(s;\mu_{0,a}).
		\end{align*}
		The weighted Hölder inequality, using
		\eqref{caseADeltaDefinition}, is
		\begin{align}
			&\left(\int_{\mathscr G}
			\left|\prod_{a=1}^rL_a(\tfrac12+it)
			\prod_{j=0}^r\mathcal M_j(\tfrac12+it)
			\right|^{1/n}dt\right)^n\nonumber\\
			&\quad\leq
			I_{k_1,\ldots,k_r}(T)^{1/(2R)}
			\prod_{j=1}^r\left(
			\int_{\mathscr G}|L_j(\tfrac12+it)|^{4/d_j}
			|\mathcal M_j(\tfrac12+it)|^{4/(d_jA_j)}dt
			\right)^{d_jA_j/4}\nonumber\\
			&\qquad\times
			\left(\int_{\mathscr G}
			|\mathcal M_0(\tfrac12+it)|^\delta dt\right)^{1/\delta}.
			\label{caseAHolder}
		\end{align}
		For $t\in\mathscr G$, with $s=\tfrac12+it$,
		\Cref{factorwiseExponentialLemma} gives
		\begin{align}
			|\mathcal M_j(s)|^{4/(d_jA_j)}
			&\asymp
			|\mathcal N_{\mathscr J+1}
			(s;2\boldsymbol k-(4/d_j)\boldsymbol e_j)|^2,
			\label{caseATwistedCollapse}\\
			|\mathcal M_0(s)|^\delta
			&\asymp
			|\mathcal N_{\mathscr J+1}(s;2\boldsymbol k)|^2.
			\label{caseAResidualCollapse}
		\end{align}
		
		A direct calculation gives, for every $a$,
		\begin{equation}\label{caseAMuSum}
			\sum_{j=0}^r\mu_{j,a}=4nk_a-2.
		\end{equation}
		Therefore another factorwise use of
		\Cref{factorwiseExponentialLemma} shows that the integrand on the left
		of \eqref{caseAHolder} is comparable with
		\[
		\left|\prod_aL_a(\tfrac12+it)\right|^{1/n}
		\prod_{a=1}^r\left|
		\prod_{\ell=2}^{\mathscr J}
		\mathcal N_{a,\ell}(\tfrac12+it;2k_a-1/n)
		\right|^2.
		\]
		The total Hölder weight of the other integrals in
		\eqref{caseAHolder} is $n-1/(2R)$.
		
		\subsubsection*{Case B: at least one $A_a$ vanishes}
		
		The equality $A_h=0$ is possible precisely when $k_h=k\geq1$ and
		$k_a=1$ for every $a\neq h$.  If
		$k=1$, all the
		$k_a$ equal $1$, and \eqref{pureMixedSecondLowerBound} proves the
		claim; in this case all the $A_a$ vanish.  Suppose therefore that
		$k>1$, so that $h$ is unique, and put $A=1-k^{-1}$.  Choose an integer
		\[
		n>\max\left\{\frac12,\ \frac r2,\
		\frac1{2k}+\frac A4\sum_{j\neq h}d_j\right\},
		\qquad
		\frac1\delta=n-\frac1{2k}-
		\frac A4\sum_{j\neq h}d_j.
		\]
		For $j\neq h$ set
		\[
		\mu_{j,a}=\frac{d_jA}{2}
		\left(2k_a-\frac4{d_j}\mathbf1_{a=j}\right),
		\qquad
		\mu_{0,a}=\frac{4k_a}{\delta},
		\]
		and take
		\begin{align*}
			\mathscr A_0^{(B)}={}&
			\mathscr A_0^{(1/n)}(\boldsymbol k,n)
			\cup\{\mu_{j,a}:j\neq h,\ 1\leq a\leq r\}
			\cup\{\mu_{0,a}:1\leq a\leq r\},\\
			\mathscr E^{(B)}={}&
			\left\{\left(\mu_{j,a},\frac{2}{d_jA}\right),
			\left(\mu_{j,a},\frac1{2n}\right):j\neq h,\ 1\leq a\leq r\right\}\\
			&\cup\left\{\left(\mu_{0,a},\frac\delta2\right),
			\left(\mu_{0,a},\frac1{2n}\right):1\leq a\leq r\right\}.
		\end{align*}
		Choose the constants and define $\mathcal M_j,\mathcal M_0$ as in
		Case~A, using these parameter sets and omitting $j=h$.  H\"older gives
		\begin{align}
			&\left(\int_{\mathscr G}
			\left|\prod_aL_a(\tfrac12+it)
			\mathcal M_0(\tfrac12+it)
			\prod_{j\neq h}\mathcal M_j(\tfrac12+it)
			\right|^{1/n}dt\right)^n\nonumber\\
			&\quad\leq I_{k_1,\ldots,k_r}(T)^{1/(2k)}\nonumber\\[-1mm]
			&\qquad\times\prod_{j\neq h}\left(
			\int_{\mathscr G}|L_j(\tfrac12+it)|^{4/d_j}
			|\mathcal M_j(\tfrac12+it)|^{4/(d_jA)}dt
			\right)^{d_jA/4}
			\left(\int_{\mathscr G}
			|\mathcal M_0(\tfrac12+it)|^\delta dt\right)^{1/\delta}.
			\label{caseBHolder}
		\end{align}
		On $\mathscr G$, the same factorwise exponential-transfer argument gives,
		for $j\ne h$,
		\[
		|\mathcal M_j(s)|^{4/(d_jA)}\asymp
		|\mathcal N_{\mathscr J+1}
		(s;2\boldsymbol k-(4/d_j)\boldsymbol e_j)|^2,
		\qquad
		|\mathcal M_0(s)|^\delta\asymp
		|\mathcal N_{\mathscr J+1}(s;2\boldsymbol k)|^2.
		\]
		Moreover,
		\[
		\sum_{j\neq h}\mu_{j,h}+\mu_{0,h}=4nk-2,
		\qquad
		\sum_{j\neq h}\mu_{j,a}+\mu_{0,a}=4n-2\quad(a\neq h).
		\]
		The total H\"older weight of the other integrals is
		\[
		\frac A4\sum_{j\neq h}d_j+\frac1\delta
		=n-\frac1{2k}.
		\]
		These identities identify the integrand on the left of
		\eqref{caseBHolder} with the same expression as in Case~A, up to
		fixed positive factors.
		
		\subsubsection*{An upper bound for the exceptional factors}
		
		To obtain the lower estimate for the left-hand integrals above, we
		bound the factors at and after the first bad block for each $L_a$.
		Fix a positive
		integer $n$ and $c_1,\ldots,c_r>0$, put
		\[
		\boldsymbol c=(c_1,\ldots,c_r),\qquad
		F(s)=\prod_{a=1}^rL_a(s),\qquad
		\boldsymbol b=2\boldsymbol c-\frac1n\boldsymbol1,
		\]
		and suppose that
		\begin{equation}\label{oneOverNConditions}
			n>\max\left\{(2\min_a c_a)^{-1},
			\frac14\sum_{a=1}^r d_a\right\}.
		\end{equation}
		Put
		\[
		\vartheta_i=\frac{d_i}{4n}\quad(1\leq i\leq r),
		\qquad
		\vartheta_0=1-\sum_{i=1}^r\vartheta_i>0,
		\]
		and
		\[
		\boldsymbol\beta^{(0)}=2\boldsymbol c,
		\qquad
		\boldsymbol\beta^{(i)}
		=2\boldsymbol c-\frac4{d_i}\boldsymbol e_i
		\quad(1\leq i\leq r).
		\]
		For $\boldsymbol\nu\in\{2,\ldots,\mathscr J+1\}^r$, put
		\[
		R_{\boldsymbol\nu}(t)=
		\prod_{a:\nu_a\leq\mathscr J}
		\prod_{\nu_a\leq j\leq\mathscr J}
		|\mathcal N_{a,j}(\tfrac12+it;b_a)|^2.
		\]
		Take $\mathscr A\supseteq\mathscr A_0^{(1/n)}(\boldsymbol c,n)$,
		as in Cases~A and~B. Before choosing $C_{\rm cut}$, fix
		\begin{equation}\label{residualMajorantParameters}
			\begin{split}
				m_{\rm res}&=\lceil\vartheta_0^{-1}\rceil,\\
				A_{\rm res}&=\sum_{a=1}^r d_a
				\left(c_a+\frac{m_{\rm res}|b_a|}{2}\right),\\
				C_{\rm res}&=2D_0 A_{\rm res}^2+r\log2 .
			\end{split}
		\end{equation}
		In addition to \eqref{baseBadCutoffChoice}, the single choice of
		$C_{\rm cut}$ must satisfy
		\begin{equation}\label{residualCutoffChoice}
			C_{\rm cut}\geq C_{\rm res}.
		\end{equation}
		After fixing $C_{\rm cut}$, choose $\widehat Y$ sufficiently large,
		with \eqref{YDominatesParameters} and
		\begin{equation}\label{oneOverNParameterDominance}
			\widehat Y^2\geq1+2n+m_{\rm res}.
		\end{equation}
		Choose $\Lambda$ last.  Implied constants may depend on the fixed
		$\boldsymbol c,\boldsymbol L,n,\mathscr A$, and on $C_{\rm cut},\widehat Y$.
		\begin{lem}\label{oneOverNResidualMeanLemma}
			There is a constant
			$c_0=c_0(\boldsymbol c,\boldsymbol L,n,\mathscr A, \widehat{Y})>0$ such that, for
			every $\boldsymbol\nu\neq
			(\mathscr J+1,\ldots,\mathscr J+1)$,
			\begin{align}
				&\int_T^{2T}
				|\mathcal N_{\boldsymbol\nu}(\tfrac12+it;
				\boldsymbol\beta^{(0)})
				\mathcal Q_{\boldsymbol\nu}(\tfrac12+it)|^2
				R_{\boldsymbol\nu}(t)^{1/\vartheta_0}\,dt\nonumber\\
				&\qquad\ll_{\boldsymbol c,\boldsymbol L,n,\mathscr A, \widehat{Y}}
				T(\log T)^{\kappa(\boldsymbol c)}
				\exp\left(-c_0
				\sum_{\nu_a\leq\mathscr J}K_{\nu_a}\right).
				\label{oneOverNResidualMeanEstimate}
			\end{align}
		\end{lem}
		
		The proof of \Cref{oneOverNResidualMeanLemma}
		is given in Section~\ref{momentComputationsSection}.
		
		\subsubsection*{The lower estimates and conclusion}
		
		The following two lemmas supply the remaining lower bounds.
		We use the notation and parameter choices introduced above, with
		arbitrary $\boldsymbol c>0$ and $n$ satisfying
		\eqref{oneOverNConditions}.
		
		The following lemma adapts Heath--Brown's lower-bound argument
		\cite{Heath-Brown1} to our weighted product of $L$-functions.
		
		\begin{lem}\label{fullIntervalHeathBrownLemma}
			With the preceding notation and under \eqref{oneOverNConditions},
			there are constants $c_{\rm HB},C_{\rm HB}>0$, independent of the
			terminal parameter $\Lambda$, such that
			\begin{equation}\label{weightedHeathBrownComparison}
				\int_T^{2T}|F(\tfrac12+it)|^{1/n}
				|\mathcal N_{\mathscr J+1}(\tfrac12+it;\boldsymbol b)|^2dt
				\geq c_{\rm HB}x_{\mathscr J}^{-C_{\rm HB}}
				T(\log T)^{\kappa(\boldsymbol c)}.
			\end{equation}
			In particular, after $\Lambda$ has been fixed, the left-hand side is
			$\gg_{\boldsymbol c,\boldsymbol L,n,\Lambda}
			T(\log T)^{\kappa(\boldsymbol c)}$.
			If $c_1=\cdots=c_r=1$, the corresponding pure mixed second moment
			satisfies
			\begin{equation}\label{pureMixedSecondLowerBound}
				\int_T^{2T}\prod_{a=1}^r|L_a(\tfrac12+it)|^2dt
				\gg_{\boldsymbol L}T(\log T)^r.
			\end{equation}
		\end{lem}
		
		The next lemma restricts the comparison to the good set.
		
		\begin{lem}\label{modifiedHeathBrownOneOverN}
			Under the preceding hypotheses, choose $\Lambda$ sufficiently large
			after $C_{\rm cut}$ and $\widehat Y$.  Then
			\begin{align}
				&\int_{\mathscr G}|F(\tfrac12+it)|^{1/n}
				|\mathcal N_{\mathscr J+1}(\tfrac12+it;\boldsymbol b)|^2dt\nonumber\\
				&\hspace{35mm}\gg_{\boldsymbol c,\boldsymbol L,n}
				T(\log T)^{\kappa(\boldsymbol c)}.
				\label{oneOverNLowerMean}
			\end{align}
		\end{lem}
		
		The proofs of
		\Cref{fullIntervalHeathBrownLemma,modifiedHeathBrownOneOverN}
		are given in Section~\ref{momentComputationsSection}.
		
		To finish, put $M(T)=T(\log T)^{\kappa(\boldsymbol k)}$.
		In both cases, \Cref{modifiedHeathBrownOneOverN} with
		$\boldsymbol c=\boldsymbol k$ gives a lower bound $\gg M(T)$ for the
		left-hand integral.  The Taylor comparisons above and
		\Cref{unifiedMeanValuesProposition} bound each integral involving
		$\mathcal M_j$ on the right by $O(M(T))$, after extension to
		$[T,2T]$. The extra powers of $x_{\mathscr J}$ in these bounds are
		bounded independently of $T$, since $x_{\mathscr J}<e^\Lambda$ by
		\eqref{terminalScaleBounds} and $\Lambda$ is fixed. Thus \eqref{caseAHolder} and \eqref{caseBHolder}
		give, respectively,
		\[
		M(T)^n\ll I_{\boldsymbol k}(T)^{1/(2R)}M(T)^{n-1/(2R)},
		\qquad
		M(T)^n\ll I_{\boldsymbol k}(T)^{1/(2k)}M(T)^{n-1/(2k)}.
		\]
		Both imply $I_{\boldsymbol k}(T)\gg M(T)$.  When all $k_a=1$,
		use \eqref{pureMixedSecondLowerBound}.  This proves \Cref{mainThm2}.
		
	}
	
	{\section{Dirichlet-polynomial mean values and twisted moments}\label{twistedMomentsSection}
		We record the mean-value estimates used in Section~\ref{momentComputationsSection}.
		The shifted fourth-moment formulae and contour identities are given in
		Appendix~\ref{appendixClassicalTwistedMoments}.}
	\subsection{Mean values of single Dirichlet polynomials}
	The very cornerstone of computations in the theory of mean values of Dirichlet polynomials is the classical Montgomery--Vaughan theorem, which gives $L^2$-control over a general Dirichlet polynomial. More specifically, let $A(s)=\sum_{n \leq N} \frac{a(n)}{n^s}$. Then $$\int_{T}^{2T}\left|A(it)\right|^2\dt = \sum_{n\leq N} (T+O(n))|a(n)|^2.$$ 
	
	{Another classical estimate handles high powers of a prime-supported
		Dirichlet polynomial.  If $B(s)=\sum_{p\leq x}b(p)p^{-s}$ and
		$x^\ell\leq T^{1-\varepsilon}$, Montgomery--Vaughan and the multinomial
		theorem give
		\[
		\int_T^{2T}|B(it)|^{2\ell}dt
		\ll T\ell!\left(\sum_{p\leq x}|b(p)|^2\right)^\ell.
		\]
	}
	\subsection{\texorpdfstring{$\textup{GL}(1)$}{GL(1)} twisted fourth moments}
	{Throughout this subsection, $\Phi$ is a fixed smooth non-negative
		function supported in $[1/2,4]$, equal to $1$ on $[1,2]$, and satisfying
		$\Phi^{(j)}(x)\ll_j1$ for every $j\geq0$.  Furthermore, let
		$A(s)=\sum_{n\leq T^{\vartheta}}a(n)n^{-s}$, where
		$a(n)\ll_\varepsilon n^\varepsilon$ and
		$\vartheta<\frac14-\varepsilon$.}
	
	{For the zeta case, we use
		\cite[Theorem~1.1 and the first remark thereafter]{BettinBuiLiRadziwill}.}

	{We keep the additive errors $O(T^{1-\eta})$, $\eta>0$,
		in the twisted estimates.  Section~\ref{momentComputationsSection}
		absorbs them into bounds of size $T^{1-o(1)}$.}
	\subsubsection{The Riemann zeta case}
	{For
		$\boldsymbol\alpha=(\alpha_1,\ldots,\alpha_4)$ satisfying
		$|\alpha_j|\ll(\log T)^{-1}$ for $1\leq j\leq4$, put
		\[
		\sigma_{\alpha_1,\alpha_2}(n)
		=\sum_{n_1n_2=n}n_1^{-\alpha_1}n_2^{-\alpha_2},
		\]
		\[
		A_{\alpha_1,\alpha_2,\alpha_3,\alpha_4}
		=\frac{\prod_{i=1}^{2}\prod_{j=3}^{4}
			\zeta(1+\alpha_i+\alpha_j)}
		{\zeta(2+\alpha_1+\alpha_2+\alpha_3+\alpha_4)},
		\]
		and
		\[
		B_{\alpha_1,\alpha_2,\alpha_3,\alpha_4,n}
		=\prod_{p^\nu\Vert n}
		\frac{\sum_{j\geq0}\sigma_{\alpha_1,\alpha_2}(p^j)
			\sigma_{\alpha_3,\alpha_4}(p^{j+\nu})p^{-j}}
		{\sum_{j\geq0}\sigma_{\alpha_1,\alpha_2}(p^j)
			\sigma_{\alpha_3,\alpha_4}(p^j)p^{-j}}.
		\]
		The Euler products in this subsection are understood by meromorphic
		continuation in a neighborhood of $\boldsymbol0$.  Finally set
		\[
		Z_{\alpha_1,\alpha_2,\alpha_3,\alpha_4,n,m}
		=A_{\alpha_1,\alpha_2,\alpha_3,\alpha_4}
		B_{\alpha_1,\alpha_2,\alpha_3,\alpha_4,n}
		B_{\alpha_3,\alpha_4,\alpha_1,\alpha_2,m}.
		\]
	}
	\begin{lem}\label{twistedFourthMomentUpperBoundMultiplicative}
		Let $$G(z_1,z_2,z_3,z_4)\coloneq \sum_{n,m} \frac{a(n)\overline{a(m)}}{[m,n]} B_{z_1,z_2,z_3,z_4,\frac{n}{(m,n)}}B_{z_3,z_4,z_1,z_2,\frac{m}{(m,n)}}.$$ Then {
			$$\int_{\mathbb{R}} |\zeta(\tfrac12+it)|^4|A(\tfrac12+it)|^2 \Phi\left(\frac{t}{T}\right)\dt \ll T(\log T)^4 \max_{\substack{|z_j|=3^j/\log T, \\ 1\leq j \leq 4}} |G(z_1,z_2,z_3,z_4)|+T^{1-\eta}$$
			for some $\eta=\eta(\varepsilon)>0$.}
	\end{lem}
	{Its proof is given in
		Appendix~\ref{appendixClassicalTwistedMoments}.}
	
	\subsubsection{The Dirichlet $L$-function case}
	{Let $\chi$ be a primitive Dirichlet character modulo
		$q$, and keep the small-shift condition
		$|\alpha_j|\ll(\log T)^{-1}$ for $1\leq j\leq4$. Write
		$A_{\boldsymbol\alpha}:=A_{\alpha_1,\alpha_2,\alpha_3,\alpha_4}$
		for the factor defined in the zeta case, and put
		\[
		S_p(\boldsymbol\alpha)=\sum_{j\geq0}
		\frac{\sigma_{\alpha_1,\alpha_2}(p^j)
			\sigma_{\alpha_3,\alpha_4}(p^j)}{p^j}
		\]
		and define, meromorphically near $\boldsymbol0$,
		\[
		A^{\chi}_{\boldsymbol\alpha}
		=A_{\boldsymbol\alpha}\prod_{p\mid q}
		S_p(\boldsymbol\alpha)^{-1}.
		\]
		This factor depends on $\chi$ only through $q$. For
		$\Re(\alpha_i+\alpha_j)>0$ ($i\leq2<j$),
		$A_{\boldsymbol\alpha}=\prod_p S_p(\boldsymbol\alpha)$, so deleting
		the factors at primes dividing $q$ gives
		$A^\chi_{\boldsymbol\alpha}=\prod_{p\nmid q}S_p(\boldsymbol\alpha)$.}
	Furthermore, we shall denote the twist of $\sigma_{\alpha_1,\alpha_2}$ by a Dirichlet character by $\sigma^{\chi}_{\alpha_1,\alpha_2}(n)\coloneq \sigma_{\alpha_1,\alpha_2}(n)\chi(n)$. Similar to the zeta case, we define $$B^{\chi}_{\alpha_1,\alpha_2,\alpha_3,\alpha_4,n} \coloneq \prod_{p^\nu \mid\mid n} \frac{\sum_{j\geq 0} \sigma^{\chi}_{\alpha_1,\alpha_2}(p^j)\sigma^{\overline{\chi}}_{\alpha_3,\alpha_4}(p^{j+\nu})p^{-j}}{\sum_{j\geq 0} \sigma^{\chi}_{\alpha_1,\alpha_2}(p^j)\sigma^{\overline{\chi}}_{\alpha_3,\alpha_4}(p^{j})p^{-j}}.$$ 
	
	{Put
		\[
		Z^\chi_{\alpha_1,\alpha_2,\alpha_3,\alpha_4,n,m}
		=A^\chi_{\alpha_1,\alpha_2,\alpha_3,\alpha_4}
		B^\chi_{\alpha_1,\alpha_2,\alpha_3,\alpha_4,n}
		B^{\overline\chi}_{\alpha_3,\alpha_4,\alpha_1,\alpha_2,m}.
		\]
	}
	
	\begin{lem}\label{twistedFourthMomentUpperBoundMultiplicativeDirichlet}
		Let $\chi$ be a fixed primitive character modulo $q$, and put
		\[
		G^{\chi}(z_1,z_2,z_3,z_4)
		=\sum_{n,m}\frac{a(n)\overline{a(m)}}{[m,n]}
		B^{\chi}_{z_1,z_2,z_3,z_4,\frac{n}{(m,n)}}
		B^{\overline\chi}_{z_3,z_4,z_1,z_2,\frac{m}{(m,n)}}.
		\]
		Then
		\begin{align*}
			&\int_{\mathbb R}|L(\tfrac12+it,\chi)|^4
			|A(\tfrac12+it)|^2\Phi(t/T)\,dt\\
			&\quad\ll_{\chi,\varepsilon}
			T(\log T)^4
			\max_{\substack{|z_j|=3^j/\log T\\1\leq j\leq4}}
			|G^\chi(z_1,z_2,z_3,z_4)|+T^{1-\eta},
		\end{align*}
		where $\eta=\eta(\varepsilon,\vartheta)>0$ depends on the fixed gap
		between $\vartheta$ and $1/4$.
	\end{lem}
	
	{Its justification is given in
		Appendix~\ref{appendixClassicalTwistedMoments}.}
	
	\subsection{Twisted second moments for primitive holomorphic cusp forms}\label{holomorphicCuspTwistedMomentSection}
	{Keep the weight $\Phi$ from the preceding subsection.
		Let $f$ be a primitive holomorphic cusp form of weight
		$\kappa\geq1$, level $N_f$, and unitary nebentypus $\omega_f$.  We use
		the unitary normalization, so the critical line is $\Re(s)=1/2$.  Its
		$L$-function is}
	\[
	L(s,f)=\sum_n \frac{\lambda_f(n)}{n^s}
	=\prod_{p} \frac{1}{1-\lambda_f(p)p^{-s}+\omega_f(p)p^{-2s}}.
	\]
	{See \cite[Section 5.11]{IwaniecKowalski} for further
		details and the normalization convention.}
	
	If $(n,m)=1$ and $(nm,N_f)=1$ we define the quantity $$Z_{s,n,m}^f = \prod_{p^a \| n} \left(\frac{\overline{\lambda_f(p^a)}-\overline{\omega_f(p)}p^{-s}\overline{\lambda_f(p^{a-2})}}{1+p^{-s}}\right) \prod_{p^b \| m} \left(\frac{\lambda_f(p^b)-\omega_f(p)p^{-s}\lambda_f(p^{b-2})}{1+p^{-s}}\right)$$ 
	We use the convention that $\lambda_f(p^{-1})=0$ in this formula.
	{Put}
	$$L^{\sharp}(s,f\times \widetilde{f})=\sum_{n\geq 1} \frac{|\lambda_f(n)|^2}{n^s},$$ where $\widetilde{f}$ is the contragredient to $f$; see Appendix~\ref{appendixRankinSelberg} for the relevant Rankin--Selberg background.
	{For later use we normalize the reduced local factor by
		\begin{equation}\label{normalizedCuspLocalFactor}
			\widehat Z^f_{z_1,z_2,n,m}
			:=(nm)^{-(z_1-z_2)/2}Z^f_{1+z_1-z_2,n,m}.
	\end{equation}}
	
	\begin{lem}\label{twistedSecondMomentUpperBoundMultiplicativeHolomorphicCusp}
		Let $f$ be a primitive holomorphic cusp form of weight $\kappa\geq1$,
		level $N_f$, and unitary nebentypus $\omega_f$. Assume the normalization convention {fixed immediately above}. Let
		\[
		A(s)=\sum_{n\leq T^{\vartheta}}\frac{a(n)}{n^s},
		\qquad
		a(n)\ll_\varepsilon n^\varepsilon,
		\]
		where $a(n)=0$ whenever $(n,N_f)>1$. Furthermore assume $\vartheta \leq \frac{1}{10}$. Define
		\[
		G^f(z_1,z_2)
		=
		\sum_{{n,m\leq T^{\vartheta}}}
		\frac{{a(n)\overline{a(m)}}}{[n,m]}
		{\widehat Z^f_{z_1,z_2,
				\frac{n}{(n,m)},\frac{m}{(n,m)}}}.
		\]
		Then
		\begin{align*}
			&\int_{\mathbb R}
			\left|L\left(\tfrac12+it,f\right)\right|^2
			\left|A\left(\tfrac12+it\right)\right|^2
			\Phi\left(\frac tT\right)\,dt\\
			&\quad\ll_{f,\varepsilon}
			T\log T
			\max_{\substack{|z_1|=2/\log T\\|z_2|=4/\log T}}
			\left|G^f(z_1,z_2)\right|{+T^{1-\eta}}.
		\end{align*}
		{Here $\theta$ is an admissible Ramanujan exponent for the
			spectral Maa\ss{} forms; one may take $\theta=7/64$ \cite{KimSarnak}.
			Any fixed $0<\eta<\frac12-\theta-2\vartheta$ is admissible after
			the coefficient-bound $\varepsilon$ loss.}
	\end{lem}
	{For trivial nebentypus, the off-diagonal estimate is
		\cite[Proposition~3.4]{AndersenThorner}.  For general nebentypus,
		including weight one, use
		\cite[Theorem~2 and Remark~12, with the erratum]{BlomerHarcosErratum}
		with both input representations equal to that of $f$.
		The coefficient product is then
		$\lambda_f(a)\overline{\lambda_f(b)}$, and the output spectrum has
		trivial central character, so Andersen--Thorner's Proposition~5.5
		follows by the same proof.  The Hecke relations give the diagonal factors
		$L^\sharp Z^f$.  Opening $|A|^2$ and expressing the two diagonal
		terms as residues on the stated circles, with their poles cancelled,
		gives the bound above.  The summed error is
		$O_{f,\varepsilon}(T^{1/2+\theta+2\vartheta+\varepsilon})$.}
	
	\subsection{Moments of products of Dirichlet polynomials} 
	{We shall often twist by a product of Dirichlet polynomials.  Let
		$D_j(s)=\sum_{n\in\mathcal D_j}a_j(n)n^{-s}$ for $1\leq j\leq R$, with
		$a_j(n)\ll_\varepsilon n^\varepsilon$, extended by zero off
		$\mathcal D_j$. Assume $|a_j(1)|\leq1$ and that the coefficient bounds
		are uniform in $j$ and $R$. In every application the prime
		supports are disjoint, equivalently
		\[
		n\in\mathcal D_i,\ m\in\mathcal D_j,\ i\neq j
		\quad\Longrightarrow\quad(n,m)=1.
		\]
		The factorization $n=n_1\cdots n_R$ with $n_j\in\mathcal D_j$ is
		unique when it exists, and at most $\omega(n)$ of the $n_j$ exceed $1$.
		Thus the coefficients $a(n)$ of $\prod_jD_j(s)$ satisfy, uniformly in $R$,
		\[
		|a(n)|\leq C_\varepsilon^{\omega(n)}n^{\varepsilon/2}
		\ll_\varepsilon n^\varepsilon,
		\]
		where $C_\varepsilon\geq1$ is the common coefficient constant for
		exponent $\varepsilon/2$ and $\omega(n)$ counts the distinct prime
		divisors of $n$.}
	
	\begin{lem}\label{splittingDirichletPolynomialsLemma}
		Let $D_1(s),\dots,D_R(s)$ be as defined above. Assume that $$\max_{n_1 \in \mathcal{D}_1, \dots, n_R \in \mathcal{D}_R} n_1\cdots n_R \leq T^{1-\varepsilon}.$$ Then {$$\frac{1}{T}\int_{T}^{2T} \left|\prod_{j=1}^R D_j(\tfrac12+it)\right|^2\, \textup{d}t = (1+o(1))\prod_{j=1}^{R} \frac{1}{T}\int_T^{2T} |D_j(\tfrac12+it)|^2\, \textup{d}t. $$}
	\end{lem}
	
	\begin{proof}
		Let $$\prod_{j=1}^R D_j(s)=\sum_{n} \frac{a(n)}{n^s}.$$ Consider the map $$\Psi: \mathcal{D}_1\times \cdots \times \mathcal{D}_R \to \mathbb{N}, (n_1,\dots,n_R)\mapsto n_1\cdots n_R.$$ Because of the $\gcd$-assumption mentioned above, the fundamental theorem of arithmetic implies that this map is injective. Therefore {$$\sum_{n} \frac{|a(n)|^2}{n} = \sum_{n_1 \in \mathcal{D}_1,\dots ,n_R \in \mathcal{D}_R} \frac{|a_1(n_1)\cdots a_R(n_R)|^2}{n_1\cdots n_R} = \prod_{j=1}^R \sum_{n_j\in \mathcal D_j} \frac{|a_j(n_j)|^2}{n_j}.$$} By the Montgomery--Vaughan mean value theorem applied twice we thus get
		\begin{align*}
			\frac{1}{T}\int_{T}^{2T} \left|\prod_{j=1}^R D_j(\tfrac12+it)\right|^2\, \textup{d}t &= (1+o(1)) \sum_{n} \frac{|a(n)|^2}{n} \\
			&= (1+o(1))\prod_{j=1}^{R} \sum_{n_j \in \mathcal{D}_j} \frac{|a_j(n_j)|^2}{n_j} \\
			&= (1+o(1))\prod_{j=1}^R \frac{1}{T} \int_T^{2T} |D_j(\tfrac12+it)|^2\,\textup{d}t.
		\end{align*}
	\end{proof}
	
	The following lemma extends Soundararajan's high powers lemma
	\cite{Sound1} to products with almost orthogonal coefficients,
	allowing us to estimate high powers of several prime sums together
	with truncated exponentials.
	
	\begin{lem}\label{highPowersAndSomeSquaresUntwistedLemma}
		{Throughout this statement $r,\mu,\nu$ and the constants in
			the coefficient bounds are fixed.} Let $\mathcal{P}$ be a finite set of primes and let $$\mathcal{P}_{j}(s)=\sum_{p \in \mathcal{P}} \frac{b_j(p)}{p^s},$$ for $1\leq j \leq r$. Let $\mu,\nu$ be non-negative integers such that $\mu+\nu=r$. For $1\leq j \leq \nu$ $$\mathcal{N}_j(s)=\sum_{k=0}^{X_j} \frac{(\mathcal{P}_{\mu+j}(s))^k}{k!} = \sum_{\substack{p \mid n \implies p \in \mathcal{P}, \\ \Omega(n)\leq X_j}} \frac{c_j(n)}{n^s},$$ where $X_j$ is a natural number. 
		
		Assume the coefficients satisfy, for $i\neq j$, an almost-orthogonality relation $$V_{i,j}=\sum_{p \in \mathcal{P}} \frac{b_{i}(p)\overline{b_j(p)}}{p}=O(1),$$ where the constant is independent of $\mathcal{P}$. For $i=j$, we let $$V_j = \sum_{p \in \mathcal{P}} \frac{|b_j(p)|^2}{p}.$$
		
		Finally let $\ell_1,\dots,\ell_\mu$ be positive integers such that $\ell_j \ll V_j$. Then, if the length of the Dirichlet polynomial $\prod_{1\leq j \leq \mu} \mathcal{P}_j^{\ell_j} \prod_{1\leq j \leq \nu} \mathcal{N}_{j}$ is $\ll T$, we have $$\int_{T}^{2T} \prod_{j=1}^\mu \left|\mathcal{P}_j(\tfrac12+it)\right|^{2\ell_j} \prod_{j=1}^{\nu} |\mathcal{N}_j(\tfrac12+it)|^2 \dt \ll T\left(\prod_{j=1}^\mu \ell_j! V_j^{\ell_j}\right)\left( \prod_{j=1}^{\nu} \exp\left( V_{j+\mu}\right)\right).$$ 
	\end{lem}
	
	{The proof is given in
		Appendix~\ref{appendixDirichletPolynomialKernels}.}
	
	\subsection{Moments of $L$-functions twisted by a product of Dirichlet polynomials}
	{For $n_j,m_j\in\mathcal D_j$ ($1\leq j\leq R$), the disjointness
		of the prime supports gives $(n_i m_i,n_j m_j)=1$ whenever $i\ne j$. Hence
		\[
		[n_1\cdots n_R,m_1\cdots m_R]=\prod_{\ell=1}^R[n_\ell,m_\ell],
		\]
		and the reduced local factors $B,B^\chi,\widehat Z^f$ split in an analogous way.
		Using these observations and applying the single-polynomial estimates in
		\Cref{twistedFourthMomentUpperBoundMultiplicative,twistedFourthMomentUpperBoundMultiplicativeDirichlet,twistedSecondMomentUpperBoundMultiplicativeHolomorphicCusp}
		to $\prod_\ell D_\ell$, we obtain the following three conclusions.}
	\begin{lem}\label{zetaFourthMomentProductSplitting}
		Let $D_1(s),\dots,D_R(s)$ be as defined in the beginning of this section. Assume $$\max_{n_1 \in \mathcal{D}_1,\dots,n_R \in \mathcal{D}_R} n_1\cdots n_R \leq T^{1/5}.$$
		Then
		\begin{align*}
			&\int_T^{2T} |\zeta(\tfrac12+it)|^4\left|\prod_{j=1}^R D_j(\tfrac12+it)\right|^2 \, \textup{d}t \\
			\ll \,\, &{T(\log T)^4\prod_{\ell=1}^R \max_{\substack{|z_j|=3^j/\log T, \\ 1\leq j \leq 4}} \left|\sum_{n,m\in \mathcal{D}_\ell} \frac{a_\ell(n)\overline{a_\ell(m)}}{[n,m]}B_{z_1,z_2,z_3,z_4,\frac{n}{(m,n)}}B_{z_3,z_4,z_1,z_2,\frac{m}{(m,n)}}\right|+T^{1-\eta}}
		\end{align*}
		{for some fixed $\eta>0$.}
	\end{lem}  
	
	Mutatis mutandis we also obtain
	
	\begin{lem}\label{dirichletFourthMomentProductSplitting}
		Let $\chi$ be a fixed primitive Dirichlet character.  Let $D_1(s),\dots,D_R(s)$ be as before and assume $$\max_{n_1 \in \mathcal{D}_1,\dots,n_R \in \mathcal{D}_R} n_1\cdots n_R \leq T^{1/5}.$$ Then 
		\begin{align*}
			&\int_T^{2T} |L(\tfrac12+it,\chi)|^4\left|\prod_{j=1}^R D_j(\tfrac12+it)\right|^2 \, \textup{d}t \\
			\ll \,\, &{T(\log T)^4\prod_{\ell=1}^R \max_{\substack{|z_j|=3^j/\log T, \\ 1\leq j \leq 4}} \left|\sum_{n,m\in \mathcal{D}_\ell} \frac{a_\ell(n)\overline{a_\ell(m)}}{[n,m]}B^{\chi}_{z_1,z_2,z_3,z_4,\frac{n}{(m,n)}}B^{\overline{\chi}}_{z_3,z_4,z_1,z_2,\frac{m}{(m,n)}}\right|+T^{1-\eta}}.
		\end{align*}
		{Here $\eta>0$ is as in
			\Cref{twistedFourthMomentUpperBoundMultiplicativeDirichlet}.}
	\end{lem}
	
	\begin{lem}\label{cuspSecondMomentProductSplitting}
		Let $f$ be a primitive holomorphic cusp form. Let $D_1(s),\dots,D_R(s)$ be as before, but assume in addition that they are supported on integers coprime to $N_f$, and assume $$\max_{n_1 \in \mathcal{D}_1, \dots, n_R \in \mathcal{D}_R} n_1\dots n_R \leq T^{1/10}.$$ Then 
		\begin{align*}
			&\int_T^{2T} |L(\tfrac12+it,f)|^2\left|\prod_{j=1}^R D_j(\tfrac12+it)\right|^2 \, \textup{d}t \\
			\ll\,\, &{T(\log T)\prod_{\ell=1}^R \max_{\substack{|z_j|=2^j/\log T, \\ 1\leq j \leq 2}} \left|\sum_{n,m\in \mathcal{D}_\ell} \frac{a_\ell(n)\overline{a_\ell(m)}}{[n,m]}\widehat Z^f_{z_1,z_2,\frac{n}{(n,m)},\frac{m}{(n,m)}}\right|+T^{1-\eta}}.
		\end{align*}
		{Here one may take any fixed
			$0<\eta<\frac12-\theta-\frac15$, with the usual harmless loss from
			the coefficient bound. Here $\theta$ is an admissible constant towards the Ramanujan--Petersson conjecture for $\textup{GL}(2)$ $L$-functions. In particular Kim--Sarnak's $\theta=7/64$ \cite{KimSarnak} gives $\eta>0$.}
	\end{lem}
	
	{
		\begin{prop}\label{verificationOfTwistedLocalKernels}
			{Fix $c>0$.  Let the contour shifts satisfy
				$|z_j|\ll1/\log T$, let $p\leq T^c$ be unramified, and write
				\[
				\mathbf z=(z_1,z_2,z_3,z_4),\qquad
				\mathbf z^\vee=(z_3,z_4,z_1,z_2),\qquad
				w=z_1-z_2.
				\]
				For $A,B\in\mathbb Z_{\geq0}$, when a quadratic form with denominator
				$[n,m]$ is put in the
				normalization
				\[
				\sum_{n,m}\frac{a(n)\overline{a(m)}}{\sqrt{nm}}K(n,m),
				\]
				the local kernels furnished by the preceding moment formulae are as
				follows.  With $d_+=(A-B)_+$ and $d_-=(B-A)_+$,
				\begin{align}
					K_p^{\zeta,4}(A,B)
					&=p^{-|A-B|/2}B_{\mathbf z,p^{d_+}}
					B_{\mathbf z^\vee,p^{d_-}},
					\label{zetaFourthLocalKernel}\\
					K_p^{\chi,4}(A,B)
					&=p^{-|A-B|/2}B^\chi_{\mathbf z,p^{d_+}}
					B^{\overline\chi}_{\mathbf z^\vee,p^{d_-}},
					\label{characterFourthLocalKernel}\\
					K_p^{f,2}(A,B)
					&=p^{-|A-B|(1+w)/2}
					Z^f_{1+w,p^{d_+},p^{d_-}}.
					\label{cuspSecondLocalKernel}
				\end{align}
				Each kernel equals $1$ when $A=B$, and for $d\geq1$ its values
				$K_p(A+d,A)=\kappa^+_{p,d}$ and
				$K_p(A,A+d)=\kappa^-_{p,d}$ are independent of $A$.  Uniformly on the
				contours there is a fixed $C>0$ such that
				\begin{equation}\label{verifiedKappaBound}
					|\kappa^+_{p,d}|+|\kappa^-_{p,d}|
					\ll(d+1)C^dp^{-d/2}.
				\end{equation}
				For the fourth moments the linear terms are
				\begin{align}
					\kappa^{\zeta,4,+}_{p,1}
					&=p^{-1/2}\{p^{-z_3}+p^{-z_4}
					+O(p^{-1+O(1/\log T)})\},\nonumber\\
					\kappa^{\zeta,4,-}_{p,1}
					&=p^{-1/2}\{p^{-z_1}+p^{-z_2}
					+O(p^{-1+O(1/\log T)})\},
					\label{zetaLinearLocalTerms}\\
					\kappa^{\chi,4,+}_{p,1}
					&=\overline\chi(p)p^{-1/2}\{p^{-z_3}+p^{-z_4}
					+O(p^{-1+O(1/\log T)})\},\nonumber\\
					\kappa^{\chi,4,-}_{p,1}
					&=\chi(p)p^{-1/2}\{p^{-z_1}+p^{-z_2}
					+O(p^{-1+O(1/\log T)})\}.
					\label{characterLinearLocalTerms}
				\end{align}
				Finally, with the convention $\lambda_f(p^{-1})=0$, the cusp-form
				parameters are exactly
				\begin{align}
					\kappa^{f,2,+}_{p,d}
					&=p^{-d(1+w)/2}
					\frac{\overline{\lambda_f(p^d)}-
						\overline{\omega_f(p)}p^{-1-w}
						\overline{\lambda_f(p^{d-2})}}{1+p^{-1-w}},
					\label{cuspPlusKappa}\\
					\kappa^{f,2,-}_{p,d}
					&=p^{-d(1+w)/2}
					\frac{\lambda_f(p^d)-\omega_f(p)p^{-1-w}
						\lambda_f(p^{d-2})}{1+p^{-1-w}}.
					\label{cuspMinusKappa}
			\end{align}}
		\end{prop}
		
		{The proof is given in
			Appendix~\ref{appendixDirichletPolynomialKernels}.}
	}
	
	{The following estimate combines the prime blocks with a
		bound independent of their number.}

	{
		\begin{lem}[Uniform multiblock kernel estimate]\label{multiBlockCombinatorialLemma}
			Let $\mathcal P=\bigsqcup_{h\leq H}\mathcal P_h$ be disjoint prime
			blocks, with at most $r_0$ labels $\gamma=(a,h)$ on each block, where
			$1\leq a\leq r_0$.  Fix $r_0$ and $0\leq\theta<1/4$, and write
			$a(\gamma)=a$ and
			\[
			P_\gamma(s)=\sum_{p\in\mathcal P_h}
			\frac{b_\gamma(p)}{p^s},
			\qquad b_\gamma(p)\ll p^\theta.
			\]
			Partition the labels into fixed powers $\mathfrak F$ and truncations
			$\mathfrak T$, combining repeated powers of the same $P_\gamma$.
			Let $\ell_\gamma$ be positive integers and $X_\gamma$ nonnegative
			integers, and put
			\begin{equation}\label{multiblockPolynomialDefinition}
				\mathcal A(s)=
				\prod_{\gamma\in\mathfrak F}P_\gamma(s)^{\ell_\gamma}
				\prod_{\gamma\in\mathfrak T}E_{X_\gamma}(P_\gamma(s))
				=\sum_n\frac{a(n)}{n^s}.
			\end{equation}
			Define
			\[
			V_{\gamma,\delta}
			=\sum_{p\in\mathcal P_h}
			\frac{b_\gamma(p)\overline{b_\delta(p)}}p,
			\qquad V_\gamma=V_{\gamma,\gamma};
			\]
			with $V_{\gamma,\delta}=0$ on different blocks.  Assume
			$\ell_\gamma\ll V_\gamma$ and
			\begin{equation}\label{multiblockAggregateCovariance}
				\sum_{a\neq b}
				\left|
				\sum_{\substack{h:(a,h),(b,h)\in\mathfrak T}}
				V_{(a,h),(b,h)}
				\right|\ll1,
				\qquad
				\sum_{\substack{\gamma\neq\delta,\ a(\gamma)\neq a(\delta)\\
						\gamma\in\mathfrak F\ {\text{ or }}\ \delta\in\mathfrak F}}
				|V_{\gamma,\delta}|\ll1.
			\end{equation}
			
			Let $K=\prod_pK_p$ satisfy the kernel hypotheses of
			\Cref{twistedCombinatorialKernelLemma}, with $W_\gamma^\pm$ defined
			there.  Allow one distinguished type $a_0$; for the twisted bound,
			require at most one label $\gamma_0\in\mathfrak F$ of that type.
			For the remaining types assume
			\begin{equation}\label{multiblockAggregateLinear}
				\sum_{a\neq a_0}
				\left(
				\left|\sum_{(a,h)\in\mathfrak T}W_{(a,h)}^+\right|
				+
				\left|\sum_{(a,h)\in\mathfrak T}W_{(a,h)}^-\right|
				\right)
				+
				\sum_{\substack{\gamma\in\mathfrak F\\a(\gamma)\neq a_0}}
				(|W_\gamma^+|+|W_\gamma^-|)
				\ll1.
			\end{equation}
			For $\gamma\in\mathfrak T$, set
			\[
			S_\gamma=V_\gamma+|W_\gamma^+|+|W_\gamma^-|+1
			+\sum_{\delta\neq\gamma}
			\bigl(|V_{\gamma,\delta}|+|V_{\delta,\gamma}|\bigr),
			\]
			and suppose
			\begin{equation}\label{multiblockCutoffConditions}
				X_\gamma\geq C_1S_\gamma,
				\qquad
				\sum_{\gamma\in\mathfrak T}e^{-c_1X_\gamma}\ll1,
			\end{equation}
			where $c_1>0$ is fixed and $C_1$ is sufficiently large in terms of
			$c_1$, $r_0$, and the preceding uniform bounds.
			
			Then
			\begin{equation}\label{multiblockCoefficientBound}
				\sum_n\frac{|a(n)|^2}{n}
				\ll
				\left(\prod_{\gamma\in\mathfrak F}
				\ell_\gamma!V_\gamma^{\ell_\gamma}\right)
				\exp\left(\sum_{\gamma\in\mathfrak T}V_\gamma\right).
			\end{equation}
			For $V>0$, write
			\begin{equation}\label{multiblockDistinguishedHBound}
				\begin{aligned}
					\mathcal H_\ell(V;W^+,W^-)
					&=(\ell!)^2\sum_{j=0}^{\ell}
					\frac{V^{\ell-j}|W^+W^-|^j}{(\ell-j)!(j!)^2}\\
					&\leq \ell!V^\ell
					\exp\left(2\sqrt{\ell|W^+W^-|/V}\right).
				\end{aligned}
			\end{equation}
			Set $\Delta_0=1$ if $\gamma_0$ is absent.  Otherwise assume
			$|W_{\gamma_0}^+|\asymp|W_{\gamma_0}^-|\asymp V_{\gamma_0}$ and put
			\[
			\Delta_0=
			\frac{\mathcal H_{\ell_{\gamma_0}}
				(V_{\gamma_0};W_{\gamma_0}^+,W_{\gamma_0}^-)}
			{\ell_{\gamma_0}!V_{\gamma_0}^{\ell_{\gamma_0}}}.
			\]
			With $W_{\mathfrak T,0}^\pm
			=\sum_{(a_0,h)\in\mathfrak T}W_{(a_0,h)}^\pm$, we also have
			\begin{equation}\label{multiblockTwistedConclusion}
				\left|\sum_{n,m}
				\frac{a(n)\overline{a(m)}}{\sqrt{nm}}K(n,m)\right|
				\ll \Delta_0
				\left(\prod_{\gamma\in\mathfrak F}
				\ell_\gamma!V_\gamma^{\ell_\gamma}\right)
				\exp\left(
				\sum_{\gamma\in\mathfrak T}V_\gamma
				+\Re(W_{\mathfrak T,0}^++W_{\mathfrak T,0}^-)
				\right).
			\end{equation}
			All implied constants are uniform in the labels and independent of $H$.
		\end{lem}
		
		{The proof is given in
			Appendix~\ref{appendixDirichletPolynomialKernels}.}
	}
	
	{
		\section{Moment computations}\label{momentComputationsSection}

		We prove the estimates from Section~\ref{holderReductionsSection}.
		\subsection{Mertens estimates}
		
		\Cref{rankinSelbergPrimeSumsProposition} gives
		\begin{equation}\label{primeBlockVariances}
			P_j\asymp_r \log_jT-\log_{j+1}T,
			\qquad
			\sum_{j=2}^{\mathscr J}P_j\asymp_{\boldsymbol L,r}\log_2T,
		\end{equation}
		and, separately for every $1\leq a\leq r$,
		\begin{equation}\label{individualVarianceSum}
			\begin{aligned}
				\sum_{j=2}^{\mathscr J}\sum_{T_{j-1}<p\leq T_j}
				\frac{|\lambda_a(p)|^2}{p}
				&=\log_2T-2\log x_{\mathscr J}
				+O_{\boldsymbol L,r,\widehat Y}(1)\\
				&=\log_2T+O_{\boldsymbol L,r,\widehat Y,\Lambda}(1).
			\end{aligned}
		\end{equation}
		The second equality uses \eqref{terminalScaleBounds}. For $a\ne b$,
		\begin{equation}\label{crossBlockCovariances}
			\sum_{p\leq x}\frac{\lambda_a(p)\overline{\lambda_b(p)}}p=O_{\boldsymbol L}(1).
		\end{equation}
		For any fixed $\varepsilon_*>0$, choosing $\Lambda$ sufficiently large gives
		\begin{equation}\label{primeBlockHierarchy}
			\sum_{\ell>j}P_\ell\ll_{\boldsymbol L,r}\log_{j+1}T
			\leq\varepsilon_*P_j\qquad(2\leq j<\mathscr J).
		\end{equation}
		Indeed, $x_j=e^{x_{j+1}}$ and $P_j\asymp x_j-x_{j+1}$.
		Fix $\varepsilon_*$ small enough for the estimates below.
		\subsection{The factorwise exponential approximation}
		
		\begin{proof}[Proof of \Cref{factorwiseExponentialLemma}]
			Here $\alpha$ is real.  For
			$z=\alpha\mathcal P_{a,j}(s)/2$ we have, by
			\eqref{YDominatesParameters} and \eqref{factorwiseCutoffs},
			$|z|\leq C_{\rm cut}|\alpha|P_j/2
			\leq C_{\rm cut}\widehat YP_j/2\leq K_j/20$.  Taylor's formula and
			$m!\geq(m/e)^m$ give
			\[
			|e^z-E_{K_j}(z)|
			\leq e^{|z|}\left(\frac{e|z|}{K_j+1}\right)^{K_j+1}
			\ll e^{\Re z-K_j/4}.
			\]
			This estimate is uniform for $\alpha\in\mathscr A$, since
			\eqref{YDominatesParameters} gives $|z|\leq K_j/20$ for every
			such $\alpha$, and the remainder bound has an absolute implied constant.
			
			Squaring gives \eqref{factorwiseExponentialApproximation}; applying it
			with $\alpha$ and $c\alpha$ gives \eqref{moveRealPowerInside}.
		\end{proof}
		
		We next check the lengths of the prime-block polynomials used in
		the moment estimates. By \eqref{rankinSelbergQuantitativeMertens},
		\[
		P_2=x_2-2x_3+O_{\boldsymbol L,\widehat Y}(1),\qquad
		P_j=1+2(x_j-x_{j+1})+o(1)\quad(3\leq j\leq\mathscr J),
		\]
		with the latter error uniform in $j$. Hence $P_j\leq3x_j$ for
		sufficiently large $T$. Also,
		$K_j\leq11C_{\rm cut}\widehat YP_j$ and
		$x_{j-1}=e^{x_j}\geq2x_j$, so
		\[
		\sum_{j=2}^{\mathscr J}\frac{K_j\log T_j}{\log T}
		=\sum_{j=2}^{\mathscr J}\frac{K_j}{\widehat Yx_j^2}
		\leq33C_{\rm cut}\sum_{j=2}^{\mathscr J}\frac1{x_j}
		\leq\frac{66C_{\rm cut}}{x_{\mathscr J}}.
		\]
		Each factor $\mathcal N_{a,j}$ has length at most $T_j^{K_j}$,
		and the corresponding factor of $\mathcal Q_{\boldsymbol\nu}$ has length at most
		$T_j^{\ell_j}\leq T_j^{K_j}$. On each block the total degree in
		$\mathcal N_{\boldsymbol\nu}\mathcal Q_{\boldsymbol\nu}$ is at most
		$rK_j$. For the selected exceptional polynomials it is at most
		$r(m_{\rm res}+1)K_j$, and for the polynomial factor in the
		Heath--Brown argument it is at most $rvK_j$, where $v=2n$.
		By \eqref{oneOverNParameterDominance}, these are all at most
		$r\widehat Y^2K_j$; for the upper-bound reduction, $\widehat Y\geq1$
		suffices. Consequently the lengths are at most
		\[
		\prod_{j=2}^{\mathscr J}T_j^{r\widehat Y^2K_j}
		\leq T^{66rC_{\rm cut}\widehat Y^2/x_{\mathscr J}}
		\leq T^{66rC_{\rm cut}\widehat Y^2/\Lambda}
		\leq T^{66/10^7}<T^{1/1000}.
		\]
		Thus we may take $\varepsilon_0=1/1000$, which meets both the
		$T^{1/5}$ and $T^{1/10}$ twisted-moment restrictions.
		We also check the coefficients uniformly in the cutoffs. For a factor
		of $\mathcal Q_{\boldsymbol\nu}$ on block $j$, the coefficient at
		$n=\prod_p p^{\mathfrak e_p}$, with $\Omega(n)=\ell_j$, is
		\[
		\frac{\ell_j!}{(C_{\rm cut}P_j)^{\ell_j}}
		\prod_p\frac{\lambda_a(p)^{\mathfrak e_p}}{\mathfrak e_p!}.
		\]
		The bound $|\lambda_a(p)|\leq2$ is elementary in degree one and follows
		from Deligne \cite{Deligne} (Deligne--Serre \cite{DeligneSerre} in weight one) for the holomorphic factors.
		Since $\ell_j!\leq(C_{\rm cut}P_j)^{\ell_j}$, the absolute coefficient
		is at most $\prod_p2^{\mathfrak e_p}/\mathfrak e_p!$.
		For $\mathcal N_{a,j}(s;\alpha)$ the corresponding bound is
		$\prod_p|\alpha|^{\mathfrak e_p}/\mathfrak e_p!$.
		Each block contains a fixed number of factors, counting the fixed integer
		powers as repeated factors. The multinomial theorem therefore gives
		\[
		|a_j(n)|\leq\prod_{p^{\mathfrak e}\Vert n}
		\frac{M^{\mathfrak e}}{\mathfrak e!}
		\leq\prod_{p^{\mathfrak e}\Vert n}
		\binom{M+\mathfrak e-1}{\mathfrak e}=d_M(n),
		\]
		where $M$ is a fixed positive integer depending only on the reduction
		parameters and $d_M$ is the $M$-fold divisor function. Block constant
		coefficients are $0$ or $1$. Since the prime supports are disjoint,
		the combined coefficient at $n=\prod_j n_j$ is bounded by
		\[
		\prod_j d_M(n_j)=d_M(n)\ll_{M,\varepsilon}n^\varepsilon,
		\]
		with the same $M$ for every block, independently of $K_j$, $\ell_j$,
		and $\mathscr J$.
		\subsection{The mean-value estimates}
		
		\begin{proof}[Proof of \Cref{unifiedMeanValuesProposition}]
			Insert the weight $\Phi(t/T)\geq\mathbf1_{[T,2T]}(t)$ in the twisted
			estimates.  Fix the twisted index $i$ and put $u=4/d_i$.
			In \Cref{multiBlockCombinatorialLemma}, use the truncated factors
			with $2\leq j<\nu_a$, namely
			\[
			E_{K_j}\left(\sum_{T_{j-1}<p\leq T_j}
			\frac{b_{a,j}(p)}{p^s}\right),
			\qquad b_{a,j}(p)=\frac{\alpha_a}{2}\lambda_a(p).
			\]
			For $\nu_a\leq\mathscr J$, use the fixed power
			$\mathcal P_{a,\nu_a}^{\ell_{\nu_a}}$; its scalar factor
			$(C_{\rm cut}P_{\nu_a})^{-\ell_{\nu_a}}$ is kept outside the combinatorial sum.
			The corresponding diagonal sums are
			\begin{align}
				V_{a,j}^{(\alpha)}
				&=\frac{\alpha_a^2}{4}
				\sum_{T_{j-1}<p\leq T_j}\frac{|\lambda_a(p)|^2}{p}
				=\frac{\alpha_a^2}{4}P_j+O_{\boldsymbol L,\boldsymbol\alpha}(1),
				\label{goodBlockVarianceApplication}\\
				V_{a,j}^{(\mathrm{bad})}
				&=\sum_{T_{j-1}<p\leq T_j}
				\frac{|\lambda_a(p)|^2}{p}=P_j+O_{\boldsymbol L}(1).
				\label{badBlockVarianceApplication}
			\end{align}
			For different blocks, $V_{\gamma,\delta}=0$. For distinct types,
			sum $V_{(a,j),(b,j)}$ over their common truncated prefix and
			use \eqref{crossBlockCovariances}.  Terms involving a fixed power satisfy
			\[
			\sum_j|V_{(a,j),(b,j)}|
			\ll_{\mathscr A}\sum_j|R_{ab,j}|\ll1\qquad(a\ne b)
			\]
			by \eqref{rankinSelbergBlockUniformity}.  This verifies
			\eqref{multiblockAggregateCovariance}.
			For each nondistinguished type $a\ne i$, the local kernel formulae and
			\eqref{crossBlockCovariances} give
			\[
			\sum_{2\leq j<\nu_a}W_{(a,j)}^\pm=O(1).
			\]
			The shifts $|z|\ll1/\log T$ change the complete prefix sum by $O(1)$, since
			\[
			\sum_{p\leq T_{\mathscr J}}\frac{|p^{-z}-1|}{p}
			\ll\frac1{\log T}\sum_{p\leq T_{\mathscr J}}\frac{\log p}{p}
			\ll\frac{\log T_{\mathscr J}}{\log T}\ll1.
			\]
			
			The linear sums in the fixed-power rows satisfy
			\[
			\sum_{\substack{\gamma\in\mathfrak F\\a(\gamma)\neq i}}
			(|W_\gamma^+|+|W_\gamma^-|)\ll1,
			\]
			by the local kernel formulae, cross Rankin--Selberg orthogonality, and
			the fact that there are at most $r$ such rows. Together these bounds
			verify \eqref{multiblockAggregateLinear}. If
			$\nu_i\leq\mathscr J$, the unique fixed-power row of distinguished type
			satisfies
			\[
			W_{i,\nu_i}^{\pm}
			=\frac{u}{2}V_{i,\nu_i}^{(\mathrm{bad})}+O(1),
			\qquad
			|W_{i,\nu_i}^{\pm}|\asymp V_{i,\nu_i}^{(\mathrm{bad})};
			\]
			if $\nu_i=\mathscr J+1$, there is no fixed-power row of type $i$.
			Thus the two alternatives in the twisted multiblock conclusion apply
			exactly.
			
			The same bounds give
			\begin{equation}\label{kernelCutoffScale}
				C_1S_\gamma\ll_{\boldsymbol L,r,\mathscr A,C_{\rm cut}}P_j
				\qquad(\gamma=(a,j)).
			\end{equation}
			Choose the marking radius and $C_1$ from the fixed data of
			\Cref{multiBlockCombinatorialLemma}, then choose $\widehat Y$ so that
			$K_j\geq C_1S_\gamma$.  The implied constant is independent of
			the block endpoints $T_j$ and truncation orders $K_j$.  Moreover,
			\[
			\sum_{(a,j)\in\mathfrak T}e^{-cK_j}
			\leq r\sum_j e^{-cK_j}\ll e^{-c'x_{\mathscr J}},
			\]
			since $K_j\gg\widehat Yx_j$ and $x_{j-1}=e^{x_j}$.
			Thus \eqref{multiblockCutoffConditions} holds.
			We first apply the coefficient estimate
			\eqref{multiblockCoefficientBound} in
			\Cref{multiBlockCombinatorialLemma}, followed by the
			Montgomery--Vaughan mean-value theorem.  Equations
			\eqref{goodBlockVarianceApplication}--
			\eqref{badBlockVarianceApplication} give
			\begin{align}
				&\int_T^{2T}
				|\mathcal N_{\boldsymbol\nu}(\tfrac12+it;\boldsymbol\alpha)
				\mathcal Q_{\boldsymbol\nu}(\tfrac12+it)|^2dt\nonumber\\
				&\quad\ll T
				\exp\left(\frac14\sum_{a=1}^r\alpha_a^2
				\sum_{2\leq j<\nu_a}
				\sum_{T_{j-1}<p\leq T_j}
				\frac{|\lambda_a(p)|^2}{p}+O(1)\right)
				\prod_{\nu_a\leq\mathscr J}
				\frac{\ell_{\nu_a}!\,(2P_{\nu_a})^{\ell_{\nu_a}}}
				{(C_{\rm cut}P_{\nu_a})^{2\ell_{\nu_a}}}.
				\label{untwistedCombinatorialApplication}
			\end{align}
			Here $2P_j$ bounds the sum in
			\eqref{badBlockVarianceApplication}.  Since $\ell_j\leq C_{\rm cut}P_j$, Stirling's
			formula gives, uniformly in $j$,
			\begin{align}
				\frac{\ell_j!\,(2P_j)^{\ell_j}}{(C_{\rm cut}P_j)^{2\ell_j}}
				&\ll \sqrt{\ell_j}
				\left(\frac{2\ell_j}{eC_{\rm cut}^{\,2}P_j}\right)^{\ell_j}
				\ll e^{-c_0\ell_j}
				\ll e^{-c_1P_j}
				\ll e^{-c_2K_j}.
				\label{badBlockFactorialSaving}
			\end{align}
			Here $c_2$ depends on the fixed $C_{\rm cut},\widehat Y$.
			Combining
			\eqref{individualVarianceSum},
			\eqref{untwistedCombinatorialApplication}, and
			\eqref{badBlockFactorialSaving} proves
			\eqref{unifiedUntwistedMean}.
			
			For the twisted estimate, apply the appropriate single-polynomial
			bound from Section~\ref{twistedMomentsSection} to the full product
			$\mathcal N_{\boldsymbol\nu}\mathcal Q_{\boldsymbol\nu}$, followed by
			\eqref{multiblockTwistedConclusion}.
			The base moment contributes $(\log T)^{u^2/4}$.  Put
			\[
			V_a(\boldsymbol\nu)=
			\sum_{2\leq j<\nu_a}
			\sum_{T_{j-1}<p\leq T_j}\frac{|\lambda_a(p)|^2}{p}.
			\]
			The quadratic terms contribute
			$\frac14\sum_a\alpha_a^2V_a(\boldsymbol\nu)+O(1)$.
			By \eqref{verifiedNondistinguishedW}--\eqref{verifiedDistinguishedW},
			the linear sums are $O(1)$ for $a\ne i$, while
			\begin{equation}\label{distinguishedLinearTermsApplication}
				\Re\sum_{2\leq j<\nu_i}(W_{i,j}^++W_{i,j}^-)
				=\frac{u\alpha_i}{2}V_i(\boldsymbol\nu)+O_{\boldsymbol L,
					\boldsymbol\alpha}(1).
			\end{equation}
			Consequently the exponential of
			the prime sums occurring before the factorial savings is
			\[
			\exp\left(
			\frac14\sum_a\alpha_a^2V_a(\boldsymbol\nu)
			+\frac{u\alpha_i}{2}V_i(\boldsymbol\nu)+O(1)
			\right).
			\]
			For $a\neq i$, the coefficient of $V_a(\boldsymbol\nu)$ is
			non-negative.  For the distinguished coordinate write
			\[
			\frac{\alpha_i^2}{4}+\frac{u\alpha_i}{2}
			=\frac{(\alpha_i+u)^2}{4}-\frac{u^2}{4}\geq-4.
			\]
			Thus \eqref{individualVarianceSum} bounds the ratio of the last
			exponential to
			\[
			(\log T)^{\frac14\{(u+\alpha_i)^2+
				\sum_{a\neq i}\alpha_a^2-u^2\}}
			\]
			by
			\[
			x_{\mathscr J}^{C_{\rm tw}}
			\exp\left(4\sum_{j\geq\nu_i}P_j\right),
			\]
			where the sum is empty if $\nu_i=\mathscr J+1$.
			The power of $x_{\mathscr J}$ accounts for the terminal variance
			loss; its exponent is fixed independently of $\Lambda$.
			
			For a non-distinguished first-bad factor,
			\eqref{badBlockFactorialSaving} supplies the saving.  For the
			distinguished factor, set
			$V=V_{i,j}^{(\mathrm{bad})}\leq P_j$.  The multiblock lemma gives
			\[
			\mathcal H_{\ell_j}(V;W^+,W^-)
			\ll \ell_j! V^{\ell_j}
			\exp\left(2\sqrt{\ell_j|W^+W^-|/V}\right).
			\]
			Since $W^\pm=(u/2)V+O(1)$ and $V=P_j+O(1)$, its normalized
			logarithm is at most
			\[
			-C_{\rm cut}\log(eC_{\rm cut})P_j
			+u\sqrt{C_{\rm cut}}P_j+O(\log P_j).
			\]
			Since $C_{\rm cut}\geq50$ and $u\leq4$,
			\[
			C_{\rm cut}\log(eC_{\rm cut})-4\sqrt{C_{\rm cut}}-8>0.
			\]
			By \eqref{primeBlockHierarchy},
			$\sum_{j\geq\nu_i}P_j\leq2P_{\nu_i}$.
			The displayed positive margin therefore absorbs the missing variance
			and the $O(\log P_j)$ term, leaving
			$\exp(-c'\sum_{\nu_a\leq\mathscr J}K_{\nu_a})$.
			The twisted-moment errors are $O(T^{1-\eta})$ for some fixed $\eta>0$.
			Since $\sum_{\nu_a\leq\mathscr J}K_{\nu_a}\ll\log_2T$, the required
			bound is at least $T(\log T)^{-C}=T^{1-o(1)}$, so these errors are
			absorbed uniformly in $\boldsymbol\nu$.  This proves
			\eqref{unifiedTwistedMean}.
		\end{proof}
		
		\subsection{The full-interval, exceptional, and good-set estimates}
		
		We prove \Cref{fullIntervalHeathBrownLemma,oneOverNResidualMeanLemma}
		first, and then combine them to prove
		\Cref{modifiedHeathBrownOneOverN}.  Write
		$\|\sum_m d(m)m^{-s}\|^2=\sum_m|d(m)|^2/m$.
		
		\begin{proof}[Proof of \Cref{fullIntervalHeathBrownLemma}]
			Put $v=2n$.  We apply Heath--Brown's rational-moment argument
			\cite[Sections~2--5]{Heath-Brown1}, checking the coefficient estimates
			for our polynomials.
			Set
			\[
			B(s)=F(s)\mathcal N_{\mathscr J+1}(s;\boldsymbol b)^v.
			\]
			Only in the half-plane $\Re(s)>1$, choose the root fixed by the Euler
			product and write
			\[
			A(s)=F(s)^{1/v}\mathcal N_{\mathscr J+1}(s;\boldsymbol b)
			=\sum_{m\geq1}\frac{a_v(m)}{m^s};
			\qquad A(s)^v=B(s).
			\]
			Thus no logarithm of $F$ is chosen on the critical line: $B$ is the
			single-valued function used by the analytic argument.  At every
			unramified prime in an active block,
			\begin{equation}\label{weightedRootPrimeCoefficient}
				a_v(p)=\frac1v\sum_a\lambda_a(p)
				+\sum_a\left(c_a-\frac1v\right)\lambda_a(p)
				=\sum_ac_a\lambda_a(p).
			\end{equation}
			The Taylor polynomials only involve primes up to $T_{\mathscr J}$, and
			\[
			\log_2T-\log_2T_{\mathscr J}=\log(\widehat Yx_{\mathscr J}^2).
			\]
			After exponentiation, this difference gives the powers of
			$x_{\mathscr J}$ in the coefficient bounds below. The primes at most
			$T_1$ and the omitted ramified primes affect only the implied constants.
			Take $N=T^{1/2}$, put
			$S_N(s)=\sum_{m\leq N}a_v(m)m^{-s}$, and write
			$\sigma_\eta=1/2+\eta/\log T$ and
			\[
			\mathscr L_0=\sum_{m\leq N}\frac{|a_v(m)|^2}{m},
			\qquad
			\mathscr L_\eta=\sum_{m\leq N}
			\frac{|a_v(m)|^2}{m^{1+2\eta/\log T}}.
			\]
			From \eqref{weightedRootPrimeCoefficient} and the Rankin--Selberg prime
			sums in Appendix~\ref{appendixRankinSelberg}, partial summation gives,
			for some fixed $C_0>0$ independent of $\Lambda$,
			\begin{equation}\label{HBRootCoefficientCritical}
				x_{\mathscr J}^{-C_0}(\log T)^{\kappa(\boldsymbol c)}
				\ll \mathscr L_0
				\ll x_{\mathscr J}^{C_0}
				(\log T)^{\kappa(\boldsymbol c)}.
			\end{equation}
			Here and below the constants may depend on the fixed tuple and on
			$\boldsymbol c,v,\widehat Y$, but not on $\Lambda$.
			
			To estimate the discarded Taylor terms, keep $F^{1/v}$ and replace
			all Taylor factors by their full exponentials. For one label $(a,j)$,
			replace its factor by $\exp(zb_a\mathcal P_{a,j}(s)/2)$ and write
			$g_m(z)$ for the resulting Dirichlet coefficients. Let $r_m$ be the
			sum of the terms of degree greater than $K_j$ in $g_m(z)$, evaluated
			at $z=1$. Cauchy's formula, whose tail kernel on $|z|=R>1$ is bounded
			by $R^{-K_j}/(R-1)$, and Cauchy--Schwarz give
			\[
			\sum_m\frac{|r_m|^2}{m^{2\sigma}}
			\ll R^{-2K_j}\max_{|z|=R}
			\sum_m\frac{|g_m(z)|^2}{m^{2\sigma}}
			\qquad(\sigma>1/2,\ R=e^{1/2}).
			\]
			On the selected block,
			\[
			g_p(z)=\sum_{b=1}^r c_b\lambda_b(p)
			+(z-1)\frac{b_a}{2}\lambda_a(p),\qquad
			\sum_{T_{j-1}<p\leq T_j}
			\frac{|g_p(z)|^2-|g_p(1)|^2}{p^{2\sigma}}=O(P_j)
			\]
			uniformly on this circle for $\sigma\geq1/2$.
			Prime powers $p^e$ with $e\geq2$ contribute
			$O(\sum_p p^{-2})=O(1)$ to the logarithms of the coefficient Euler
			products. Comparing these products and using $R^{-2K_j}=e^{-K_j}$ gives
			\begin{equation}\label{HBLabelledTailBound}
				\frac{\displaystyle\sum_m |r_m|^2m^{-2\sigma}}
				{\displaystyle\sum_m |g_m(1)|^2m^{-2\sigma}}
				\ll\exp(-K_j+B_{\rm HB}P_j).
			\end{equation}
			Here $B_{\rm HB}$ depends only on the fixed tuple, $\boldsymbol c$, $v$,
			and $R$, independently of the block endpoints $T_h$ and truncation
			orders $K_h$. These coefficient sums converge for $\sigma>1/2$;
			the same argument holds at $\sigma=1/2$ after restricting to a finite
			prime set. Choose $\widehat Y$ so that $K_j\geq2B_{\rm HB}P_j$.
			Then \eqref{HBLabelledTailBound} is $O(e^{-K_j/2})$.
			For a set of omitted labelled tails, multivariable Cauchy gives
			the same bound with the cutoff exponents summed. Taking square roots,
			expanding the product of the Taylor truncations, and applying the
			triangle inequality bounds the total relative coefficient error by
			\[
			\prod_{a,j}\left(1+O(e^{-K_j/4})\right)-1
			\ll\sum_j e^{-K_j/4}.
			\]
			The final choice of $\Lambda$ makes this error smaller than any fixed
			positive fraction.
			Fix a sufficiently large constant $A_0>1$ from the fixed data.
			At the critical line first restrict to the unramified primes
			\[
			p\leq\min(T_{\mathscr J},N^{1/A_0})
			\]
			and apply Rankin's trick together with \eqref{HBLabelledTailBound}
			to obtain the lower bound in \eqref{HBRootCoefficientCritical}.
			For the critical upper bound one uses
			\[
			\sum_{m\leq N}\frac{|a_v(m)|^2}{m}
			\leq e\sum_m\frac{|a_v(m)|^2}{m^{1+1/\log N}},
			\]
			and the same comparison gives, for a fixed $C_1>0$ independent of $\Lambda$,
			\begin{equation}\label{HBRootCoefficientShifted}
				\mathscr L_\eta\gg\eta^{-C_1}\mathscr L_0
				\qquad(1\leq\eta\leq x_{\mathscr J})
			\end{equation}
			Indeed, on the active primes $p\leq T_{\mathscr J}$ the change in the
			logarithm of the Euler product is
			$O(\eta/(\widehat Yx_{\mathscr J}^2))$.  Above
			$T_{\mathscr J}$ the prime coefficient is
			$v^{-1}\sum_a\lambda_a(p)$, and partial summation in
			\eqref{rankinSelbergQuantitativeMertens} shows that the logarithmic loss
			is $O(\log(2+\eta))$.  The terms of local degree at least two and the
			Taylor tails contribute $O(1)$ by the preceding comparison.
			We now apply the standard Heath--Brown comparison.  With
			$H=C_{\rm G}\sqrt{\log T}$ put
			\[
			w_T(t)=\int_{T+H}^{2T-H}e^{-(t-u)^2}\,du,
			\qquad G(s)=B(s)-S_N(s)^v,
			\]
			and define
			\begin{align*}
				\mathcal L(\sigma)&=\int_{\mathbb R}|S_N(\sigma+it)|^2w_T(t)dt,\\
				\mathcal K(\sigma)&=\int_{\mathbb R}|G(\sigma+it)|^{2/v}w_T(t)dt,\\
				\mathcal J(\sigma)&=\int_{\mathbb R}|B(\sigma+it)|^{2/v}w_T(t)dt.
			\end{align*}
			Apply Gabriel's theorems to
			$B(s)e^{v(s-iu)^2/2}$ and $G(s)e^{v(s-iu)^2/2}$, and the square-mean
			estimate to $S_N(s)e^{(s-iu)^2/2}$, then integrate over
			$u\in[T+H,2T-H]$.  The resulting weight is $e^{\sigma^2}w_T(t)$.
			The weighted Montgomery--Vaughan theorem and
			\eqref{HBRootCoefficientCritical}--\eqref{HBRootCoefficientShifted}
			give, with constants independent of $\Lambda$,
			\begin{equation}\label{HBWeightedRelativeL}
				\mathcal L(\tfrac12)\asymp T\mathscr L_0,
				\qquad
				\mathcal L(\sigma_\eta)\gg
				\eta^{-C_1}\mathcal L(\tfrac12)
				\quad(1\leq\eta\leq x_{\mathscr J}).
			\end{equation}
			Because the coefficients of $G$ vanish through $N$,
			\[
			G(\tfrac54+it)\ll_\varepsilon N^{-1/4+\varepsilon}.
			\]
			If $F$ contains the zeta factor, multiply $G$ and $B$ by $s-1$ before
			applying Gabriel's strip and rectangle theorems.  For each localization
			center $u\in[T+H,2T-H]$, one has $|s-1|\asymp T$ throughout the
			range $|t-u|\ll\sqrt{\log T}$; the contribution outside this range
			is negligible by the decay of $e^{-(t-u)^2}$, as below.  Thus the
			pole-removing factor introduces comparable factors $T^{2/v}$ in every
			boundary integral. Heath--Brown's $J$--$K$--$L$ argument now applies
			verbatim after normalizing all three quantities by the actual value
			$\mathcal L(1/2)$.  The relative lower bound in
			\eqref{HBWeightedRelativeL} ensures that a sufficiently large fixed
			$\eta$ may be chosen independently of $\Lambda$; the final terminal
			choice includes $\Lambda\geq\eta$, so that $\eta\leq x_{\mathscr J}$.
			The choice $N=T^{1/2}$ gives a power saving for the integral involving
			$G$ on the right boundary; Gabriel convexity propagates this saving to
			$\sigma_\eta$, and
			\[
			|x+y|^{2/v}\ll_v|x|^{2/v}+|y|^{2/v}
			\]
			compares $\mathcal L$, $\mathcal J$, and $\mathcal K$.  A final strip
			interpolation for $B$ between $1/2$ and $3/2$, with this fixed $\eta$,
			yields
			\[
			\mathcal J(\tfrac12)\gg\mathcal L(\tfrac12)
			\geq c_{\rm HB}x_{\mathscr J}^{-C_{\rm HB}}
			T(\log T)^{\kappa(\boldsymbol c)}.
			\]
			To recover the interval integral, note
			that $0\leq w_T(t)\leq\sqrt\pi$, while, for
			$t\notin[T,2T]$,
			\[
			w_T(t)\ll T
			\exp\!\left(-\{H+\operatorname{dist}(t,[T,2T])\}^2\right).
			\]
			The standard polynomial vertical-strip bounds for the fixed
			$L$-functions, together with the elementary bounds for the finite
			Dirichlet polynomials, make the contribution of this tail $O(T^{-A})$
			for any prescribed $A$, once the fixed constant
			$C_{\rm G}$ is chosen large enough.  This localization constant does
			not enter the prime-block construction.  Since
			\[
			|B(\tfrac12+it)|^{2/v}
			=|F(\tfrac12+it)|^{2/v}
			|\mathcal N_{\mathscr J+1}(\tfrac12+it;\boldsymbol b)|^2,
			\]
			we obtain \eqref{weightedHeathBrownComparison}.
			For \eqref{pureMixedSecondLowerBound}, take $v=1$, $A=B=F$,
			and let $S_N$ truncate the Dirichlet series of $F$.  By
			\eqref{rankinSelbergQuantitativeMertens},
			\[
			\sum_{p\leq x}\frac{|\sum_a\lambda_a(p)|^2}{p}
			=r\log_2x+O(1).
			\]
			Thus $\sum_{m\leq N}|a_1(m)|^2/m\asymp(\log T)^r$, and the same
			comparison, with the square-mean triangle inequality, gives
			\eqref{pureMixedSecondLowerBound}.
		\end{proof}
		
		\begin{proof}[Proof of \Cref{oneOverNResidualMeanLemma}]
			Fix $\boldsymbol\nu\neq
			(\mathscr J+1,\ldots,\mathscr J+1)$ and put
			$\nu_{\min}=\min_a\nu_a$.  Use the constants
			$m_{\rm res},A_{\rm res},C_{\rm res}$ in
			\eqref{residualMajorantParameters}, fixed before $C_{\rm cut}$ and
			$\widehat Y$.  For every $x\geq0$,
			\begin{equation}\label{residualIntegerMajorant}
				x^{2/\vartheta_0}\leq1+x^{2m_{\rm res}}.
			\end{equation}
			Apply this to the exceptional factors, obtaining sums of squared moduli
			of Dirichlet polynomials.
			Before $\nu_{\min}$ all factors belong to the common good prefix.
			The coefficient estimate \eqref{multiblockCoefficientBound}, with
			coefficients $c_a\lambda_a(p)$ and no fixed-power labels, gives
			\begin{equation}\label{residualCommonPrefixNorm}
				\left\|
				\prod_{a=1}^r\prod_{2\leq j<\nu_{\min}}
				\mathcal N_{a,j}(s;2c_a)
				\right\|^{\,2}
				\ll(\log T)^{\kappa(\boldsymbol c)}.
			\end{equation}
			Here the diagonal sum is at most $\kappa(\boldsymbol c)\log_2T+O(1)$,
			and \eqref{crossBlockCovariances} bounds the cross terms after summing
			the common prefix.
			For $j\geq\nu_{\min}$ put
			\[
			\mathcal R_j=\{a:\nu_a\leq j\},\qquad
			\mathcal F_j=\{a:\nu_a=j\},\qquad f_j=|\mathcal F_j|,
			\]
			and let
			\[
			U_j=\sum_{T_{j-1}<p\leq T_j}\frac1p,\qquad
			S_j(s)=\sum_{T_{j-1}<p\leq T_j}p^{-s},
			\]
			with the same unramified-prime convention as before.  By
			\eqref{unweightedVarianceComparison}, $U_j\leq D_0P_j$.
			There are at most $2^{|\mathcal R_j|}\leq2^r$ selections from
			\eqref{residualIntegerMajorant} on this block. The Taylor factors in
			each selected Dirichlet polynomial are
			\[
			\prod_{a:j<\nu_a}\mathcal N_{a,j}(s;2c_a)
			\prod_{a\in E}\mathcal N_{a,j}(s;b_a)^{m_{\rm res}},
			\qquad E\subseteq\mathcal R_j.
			\]
			Since $|\lambda_a(p)|\leq d_a\leq2$, the absolute coefficients of their product
			are dominated coefficientwise by those of
			$\exp(A_{\rm res}S_j(s))$, independently of $K_j$.
			Let $\mathcal A_{j,E}$ be their product with the factors of
			$\mathcal Q_{\boldsymbol\nu}$ on this block. Put $L_j=f_j\ell_j$.
			The product of these latter factors has absolute coefficients bounded by those of
			\[
			\frac{2^{L_j}S_j(s)^{L_j}}
			{(C_{\rm cut}P_j)^{L_j}}.
			\]
			For an integer $L\geq0$ and $A\geq0$, write
			\[
			S_j(s)^L e^{AS_j(s)}=\sum_n\frac{c(n)}{n^s}.
			\]
			If $n$ is supported on this block and $\Omega(n)=L+k$, then
			\[
			c(n)=\frac{(L+k)!A^k}{k!}
			\prod_{p^e\Vert n}\frac1{e!}\qquad(k\geq0),
			\]
			while $c(n)=0$ when $\Omega(n)<L$. The multinomial theorem and
			Leibniz's rule give
			\begin{align}
				\sum_n\frac{|c(n)|^2}{n}
				&\leq\sum_n\frac{|c(n)|^2}{n}\prod_{p^e\Vert n}e!
				\nonumber\\
				&=\sum_{k\geq0}\frac{(L+k)!}{(k!)^2}
				A^{2k}U_j^{L+k}\nonumber\\
				&=e^{A^2U_j}(L!)^2
				\sum_{h=0}^L
				\frac{U_j^{L-h}(A^2U_j^2)^h}{(L-h)!(h!)^2}
				\nonumber\\
				&\leq e^{A^2U_j}L!U_j^L e^{2A\sqrt{LU_j}}.
				\label{residualCoefficientBound}
			\end{align}
			The last inequality follows from $L!/(L-h)!\leq L^h$ and
			$\sum_{h\geq0}y^{2h}/(h!)^2\leq e^{2y}$.
			If $f_j\geq1$, apply \eqref{residualCoefficientBound} with $L=L_j$ and
			$A=A_{\rm res}$, multiply by $4^{L_j}(C_{\rm cut}P_j)^{-2L_j}$,
			and use $U_j\leq D_0P_j$. This bounds $\|\mathcal A_{j,E}\|^2$ by
			\[
			\exp(A_{\rm res}^2D_0P_j)
			\frac{L_j!(4D_0P_j)^{L_j}}
			{(C_{\rm cut}P_j)^{2L_j}}
			\exp\!\left(2A_{\rm res}\sqrt{L_jD_0P_j}\right).
			\]
			Use $L_j!\leq L_j^{L_j}$ and
			$L_j\leq rC_{\rm cut}P_j$, and put
			$\tau=\log(C_{\rm cut}/(4D_0r))\geq16$.
			The logarithm of this bound is at most
			\[
			A_{\rm res}^2D_0P_j-\tau L_j
			+2A_{\rm res}\sqrt{L_jD_0P_j}
			\leq 2A_{\rm res}^2D_0P_j-\frac{\tau}{2}L_j.
			\]
			Here we used
			$2A\sqrt{LU}\leq(\tau/2)L+(2/\tau)A^2U$.
			Since $C_{\rm cut}\geq2$ and $P_j\geq1$,
			$L_j\geq f_jC_{\rm cut}P_j/2$.  Therefore, after summing the
			at most $2^r$ selections,
			\begin{equation}\label{oneOverNLocalBadBlockCoefficientBound}
				\sum_{E\subseteq\mathcal R_j}
				\|\mathcal A_{j,E}\|^{\,2}
				\leq
				\exp\!\left(C_{\rm res}P_j
				-4C_{\rm cut}f_jP_j\right).
			\end{equation}
			For $f_j=0$, use the $L=0$ case of
			\eqref{residualCoefficientBound}.  The bound is independent of
			the truncation orders $K_j$, $\widehat Y$, and $\Lambda$.
			Different blocks have disjoint prime supports, so the sums of squared
			coefficients divided by their indices factor over the blocks.  Every selected polynomial has length
			at most $T^{\varepsilon_0}$ by the length calculation, which
			includes the fixed power $m_{\rm res}$.  Montgomery--Vaughan,
			\eqref{residualCommonPrefixNorm}, and
			\eqref{oneOverNLocalBadBlockCoefficientBound} give
			\begin{align*}
				&\int_T^{2T}
				|\mathcal N_{\boldsymbol\nu}(\tfrac12+it;
				\boldsymbol\beta^{(0)})
				\mathcal Q_{\boldsymbol\nu}(\tfrac12+it)|^2
				R_{\boldsymbol\nu}(t)^{1/\vartheta_0}\,dt\\
				&\qquad\ll T(\log T)^{\kappa(\boldsymbol c)}
				\exp\!\left(
				C_{\rm res}\sum_{j\geq\nu_{\min}}P_j
				-4C_{\rm cut}\sum_{\nu_a\leq\mathscr J}P_{\nu_a}
				\right).
			\end{align*}
			By \eqref{primeBlockHierarchy},
			\[
			\sum_{j\geq\nu_{\min}}P_j
			\leq\sum_{\nu_a\leq\mathscr J}\sum_{j\geq\nu_a}P_j
			\leq2\sum_{\nu_a\leq\mathscr J}P_{\nu_a}.
			\]
			Since $C_{\rm cut}\geq C_{\rm res}$ and
			$K_j\leq11C_{\rm cut}\widehat YP_j$, the last exponent is at most
			\[
			-2C_{\rm cut}\sum_{\nu_a\leq\mathscr J}P_{\nu_a}
			\leq-\frac2{11\widehat Y}\sum_{\nu_a\leq\mathscr J}K_{\nu_a}.
			\]
			This proves \eqref{oneOverNResidualMeanEstimate} with
			$c_0=2/(11\widehat Y)$, independently of $\Lambda$.
		\end{proof}
		
		\begin{proof}[Proof of \Cref{modifiedHeathBrownOneOverN}]
			Write $s=\tfrac12+it$ and
			$W(t)=|F(s)|^{1/n}|\mathcal N_{\mathscr J+1}(s;\boldsymbol b)|^2$.
			We combine \Cref{fullIntervalHeathBrownLemma} with
			\begin{equation}\label{oneOverNComplementTarget}
				\int_{[T,2T]\setminus\mathscr G}W(t)\,dt
				\leq\frac12\int_T^{2T}W(t)\,dt.
			\end{equation}
			The parameters satisfy
			\begin{equation}\label{oneOverNBalanceIdentities}
				\frac4{d_i}\vartheta_i=\frac1n,\qquad
				\vartheta_0+\sum_i\vartheta_i=1,\qquad
				\vartheta_0\boldsymbol\beta^{(0)}
				+\sum_i\vartheta_i\boldsymbol\beta^{(i)}=\boldsymbol b.
			\end{equation}
			On $\mathscr B(\boldsymbol\nu)$, apply
			\Cref{factorwiseExponentialLemma} to the good prefixes and use
			\[
			|\mathcal N_{\mathscr J+1}(s;\boldsymbol b)|^2
			=|\mathcal N_{\boldsymbol\nu}(s;\boldsymbol b)|^2R_{\boldsymbol\nu}(t),
			\qquad
			\mathbf1_{\mathscr B(\boldsymbol\nu)}\leq|\mathcal Q_{\boldsymbol\nu}(s)|^2.
			\]
			By \eqref{oneOverNBalanceIdentities} and \eqref{cutoffSummability},
			\begin{align}
				\mathbf1_{\mathscr B(\boldsymbol\nu)}(t)W(t)
				&\ll
				\left(
				|\mathcal N_{\boldsymbol\nu}(s;\boldsymbol\beta^{(0)})
				\mathcal Q_{\boldsymbol\nu}(s)|^2
				R_{\boldsymbol\nu}(t)^{1/\vartheta_0}
				\right)^{\vartheta_0}\nonumber\\
				&\quad\times\prod_{i=1}^r
				\left(
				|L_i(s)|^{4/d_i}
				|\mathcal N_{\boldsymbol\nu}(s;\boldsymbol\beta^{(i)})
				\mathcal Q_{\boldsymbol\nu}(s)|^2
				\right)^{\vartheta_i}.
				\label{oneOverNBadSetPointwise}
			\end{align}
			Integrate and apply H\"older, \Cref{oneOverNResidualMeanLemma}, and
			\eqref{unifiedTwistedMean}.  Each moment bound has logarithmic exponent
			$\kappa(\boldsymbol c)$, since
			$\boldsymbol\beta^{(i)}+(4/d_i)\boldsymbol e_i=2\boldsymbol c$.
			Thus, with constants independent of $\Lambda$,
			\[
			\int_{\mathscr B(\boldsymbol\nu)}W(t)\,dt
			\ll x_{\mathscr J}^{C_*}T(\log T)^{\kappa(\boldsymbol c)}
			\exp\left(-c\sum_{\nu_a\leq\mathscr J}K_{\nu_a}\right).
			\]
			The sum over all bad sets satisfies
			\[
			\sum_{\boldsymbol\nu\ne(\mathscr J+1,\ldots,\mathscr J+1)}
			\exp\left(-c\sum_{\nu_a\leq\mathscr J}K_{\nu_a}\right)
			=\left(1+\sum_{j=2}^{\mathscr J}e^{-cK_j}\right)^r-1
			\ll e^{-c'x_{\mathscr J}},
			\]
			using $K_j\gg\widehat Yx_j$ and $x_{j-1}=e^{x_j}$.  Hence
			\begin{equation}\label{oneOverNBadSetRemoval}
				\int_{[T,2T]\setminus\mathscr G}W(t)\,dt
				\ll x_{\mathscr J}^{C_*}e^{-c'x_{\mathscr J}}
				T(\log T)^{\kappa(\boldsymbol c)}.
			\end{equation}
			Divide by \eqref{weightedHeathBrownComparison}.  The ratio is
			$O(x_{\mathscr J}^{C_*+C_{\rm HB}}e^{-c'x_{\mathscr J}})$, uniformly
			in $\Lambda$.  Choose $\Lambda$ so that this is at most $1/2$
			for every $x_{\mathscr J}\geq\Lambda$.  This proves
			\eqref{oneOverNComplementTarget}, and therefore
			\[
			\int_{\mathscr G}W(t)\,dt
			\geq\frac12c_{\rm HB}x_{\mathscr J}^{-C_{\rm HB}}
			T(\log T)^{\kappa(\boldsymbol c)}
			\gg T(\log T)^{\kappa(\boldsymbol c)},
			\]
			since $x_{\mathscr J}<e^\Lambda$ and $\Lambda$ is fixed.
		\end{proof}
		
	}
	
	\appendix
	
	{\section{Classical twisted fourth moments and contour identities}\label{appendixClassicalTwistedMoments}}
	{
		Retain the notation and hypotheses on $A$ and $\Phi$ from
		Section~\ref{twistedMomentsSection}.
		\begin{prop}\label{bblrFourthMoment}\textup{\cite[Theorem~1.1]{BettinBuiLiRadziwill}}
			Let $T\geq2$, let $|\alpha_j|\ll(\log T)^{-1}$ for
			$1\leq j\leq4$, and put
			\[
			I_{\alpha_1,\alpha_2,\alpha_3,\alpha_4}
			=\int_{\mathbb R}\prod_{j=1}^{2}
			\zeta(\tfrac12+\alpha_j+it)
			\prod_{j=3}^{4}\zeta(\tfrac12+\alpha_j-it)
			|A(\tfrac12+it)|^2\Phi(t/T)\,dt.
			\]
			Then
			\begin{align*}
				I_{\alpha_1,\alpha_2,\alpha_3,\alpha_4}
				={}&\sum_g\sum_{(n,m)=1}
				\frac{a(gn)\overline{a(gm)}}{gnm}
				\int_{\mathbb R}\Bigg(
				Z_{\alpha_1,\alpha_2,\alpha_3,\alpha_4,n,m}\\
				&\quad+\left(\frac t{2\pi}\right)^{-\alpha_1-\alpha_2-\alpha_3-\alpha_4}
				Z_{-\alpha_3,-\alpha_4,-\alpha_1,-\alpha_2,n,m}\\
				&\quad+\left(\frac t{2\pi}\right)^{-\alpha_1-\alpha_3}
				Z_{-\alpha_3,\alpha_2,-\alpha_1,\alpha_4,n,m}\\
				&\quad+\left(\frac t{2\pi}\right)^{-\alpha_1-\alpha_4}
				Z_{-\alpha_4,\alpha_2,\alpha_3,-\alpha_1,n,m}\\
				&\quad+\left(\frac t{2\pi}\right)^{-\alpha_2-\alpha_3}
				Z_{\alpha_1,-\alpha_3,-\alpha_2,\alpha_4,n,m}\\
				&\quad+\left(\frac t{2\pi}\right)^{-\alpha_2-\alpha_4}
				Z_{\alpha_1,-\alpha_4,\alpha_3,-\alpha_2,n,m}
				\Bigg)\Phi(t/T)\,dt\\
				&+O_\varepsilon(T^{1-\eta}),
			\end{align*}
			where $\eta=\eta(\varepsilon,\vartheta)>0$ depends on the fixed gap
			between $\vartheta$ and $1/4$.
		\end{prop}
		
		The version of \cite[Theorem~1.1]{BettinBuiLiRadziwill} with the present
		weight follows from a fixed dyadic partition of $\Phi$.  Its error
		$O_\varepsilon(T^{1/2+2\vartheta+\varepsilon}
		+T^{3/4+\vartheta+\varepsilon})$ is a power saving in the range used
		here.
		
		Put
		\[
		\Sigma_{\boldsymbol\alpha}=\alpha_1+\alpha_2+\alpha_3+\alpha_4,
		\qquad
		Q_{\boldsymbol\alpha}(w)
		=(w-\alpha_1)(w-\alpha_2)(w+\alpha_3)(w+\alpha_4),
		\]
		and define
		\begin{align*}
			F_{\boldsymbol\alpha}(z_1,z_2,z_3,z_4)
			={}&\sum_{n,m}\frac{a(n)\overline{a(m)}}{[m,n]}
			Z_{z_1,z_2,z_3,z_4,\frac n{(m,n)},\frac m{(m,n)}}\\
			&\times\int_{\mathbb R}\Phi(t/T)
			\left(\frac t{2\pi}\right)^{
				(z_1+z_2+z_3+z_4-\Sigma_{\boldsymbol\alpha})/2}\,dt.
		\end{align*}
		This function is meromorphic near the origin and is symmetric separately
		in $(z_1,z_2)$ and $(z_3,z_4)$.
		
		\begin{lem}\label{snaithContourLemma}
			Assume $|\alpha_j|\leq1/\log T$ for $1\leq j\leq4$.  Then
			\begin{align*}
				I_{\alpha_1,\alpha_2,\alpha_3,\alpha_4}
				={}&O_\varepsilon(T^{1-\eta})\\
				&+\frac1{4(2\pi i)^4}
				\oint_{\substack{|z_j|=3^j/\log T\\1\leq j\leq4}}
				\frac{F_{\boldsymbol\alpha}(z_1,z_2,z_3,z_4)
					\Delta(z_1,z_2,-z_3,-z_4)^2}
				{\displaystyle\prod_{j=1}^{2}
					Q_{\boldsymbol\alpha}(z_j)Q_{\boldsymbol\alpha}(-z_{j+2})}
				\,dz_1\,dz_2\,dz_3\,dz_4,
			\end{align*}
			where $\Delta(w_1,w_2,w_3,w_4)=\prod_{i<j}(w_i-w_j)$.
		\end{lem}
		
		\begin{proof}
			This is the specialization of
			\cite[Lemma~2.5.1]{ConreyFarmerKeatingRubinsteinSnaith} with signed
			variables $(z_1,z_2,-z_3,-z_4)$.  The identity residue has exponent
			zero, the full-dual residue has exponent
			$-\Sigma_{\boldsymbol\alpha}$, and, for example, the residue
			$(-\alpha_3,\alpha_2,-\alpha_1,\alpha_4)$ has exponent
			$-\alpha_1-\alpha_3$.  These are exactly the six powers in
			\Cref{bblrFourthMoment}.
		\end{proof}
		
		\begin{proof}[Proof of \Cref{twistedFourthMomentUpperBoundMultiplicative}]
			Let $\boldsymbol\alpha\to\boldsymbol0$ in \Cref{snaithContourLemma}.
			The original integral is holomorphic in the shifts, and the fixed
			contours avoid the poles, so the denominator becomes
			$z_1^4z_2^4z_3^4z_4^4$.
			On these circles,
			\[
			\Delta(z_1,z_2,-z_3,-z_4)^2\ll(\log T)^{-12},
			\qquad A_{z_1,z_2,z_3,z_4}\ll(\log T)^4,
			\]
			and
			\[
			\int_{\mathbb R}\Phi(t/T)
			\prod_{j=1}^{4}\left(\frac t{2\pi}\right)^{z_j/2}dt\ll T.
			\]
			Trivial estimation of the contour integral gives the asserted bound.
		\end{proof}
		
		\subsection*{The fixed-character variant}
		Retain the notation $A^\chi_{\boldsymbol\alpha}$,
		$B^\chi_{\boldsymbol\alpha,n}$, and
		$Z^\chi_{\boldsymbol\alpha,n,m}$ from
		Section~\ref{twistedMomentsSection}.  A fixed-conductor shifted
		fourth-moment formula analogous to
		\cite[Theorem~1.1]{BettinBuiLiRadziwill} was communicated to the author by
		Winston Heap (private communication).  \Cref{twistedFourthMomentUpperBoundMultiplicativeDirichlet} is its
		unshifted contour consequence.  The formula follows from the same argument after
		inserting the fixed character in the shifted divisor coefficients,
		deleting the finitely many Euler factors at primes dividing the
		conductor, and replacing the archimedean scale $t/(2\pi)$ by
		$qt/(2\pi)$.  Since $q$ is fixed, neither the permitted polynomial
		length nor the power-saving error changes.

	}
	
	{\section{Combinatorial identities}\label{appendixAlgebraicTransport}}
	{This appendix collects the finite-dimensional coefficient identities used in the Dirichlet-polynomial estimates of Section~\ref{twistedMomentsSection}.  The arguments themselves are purely algebraic.}
	\begin{lem}\label{coefficientKernelBookkeepingLemma}
		{Let $\mathcal{P}$ be a finite set of primes and $\mathbf{z}=(z_p)_{p \in \mathcal{P}}$. To each $p \in \mathcal{P}$ we associate a kernel $K_p:\mathbb{Z}_{\geq0}\times \mathbb{Z}_{\geq0}\to\mathbb{C}$, and we put
			\[
			K(\boldsymbol{\alpha},\boldsymbol{\beta})
			=\prod_{p\in\mathcal P}K_p(\alpha_p,\beta_p).
			\]}
		
		Given complex numbers $u_j(p)$ and $v_h(p)$, we define, for $1\leq j \leq r$ and $1\leq h \leq q$, the objects $$L_j(\mathbf{z})=\sum_{p \in \mathcal{P}} u_j(p)z_p, \qquad M_h(\mathbf{z})=\sum_{p\in \mathcal{P}} v_h(p)z_p.$$ {Given positive integers $\ell_1,\dots,\ell_r$ and non-negative integers $X_1,\dots,X_q$, put
			\[
			E_X(z)=\sum_{m=0}^{X}\frac{z^m}{m!},
			\qquad
			F(\mathbf{z})=\prod_{j=1}^rL_j(\mathbf z)^{\ell_j}
			\prod_{h=1}^qE_{X_h}(M_h(\mathbf z))
			=\sum_{\boldsymbol\alpha}c_{\boldsymbol\alpha}
			\mathbf z^{\boldsymbol\alpha}.
			\]}
		Here $\boldsymbol{z}^{\boldsymbol{\alpha}}$ is the usual multi-index notation $\prod_{p\in \mathcal{P}} z_p^{\alpha_p}$, and the coefficients $c_{\boldsymbol{\alpha}}$ are defined implicitly. 
		{All four arrays in the following definition are required to have
			non-negative integer entries.}
		
		Finally, let
		\[
		\begin{aligned}
			\mathbf a&=(a_{j,p})_{1\leq j\leq r,\ p\in\mathcal P},
			&\widetilde{\mathbf a}&=(\widetilde a_{j,p})_{1\leq j\leq r,\ p\in\mathcal P},\\
			\mathbf b&=(b_{h,p})_{1\leq h\leq q,\ p\in\mathcal P},
			&\widetilde{\mathbf b}&=(\widetilde b_{h,p})_{1\leq h\leq q,\ p\in\mathcal P}.
		\end{aligned}
		\]
		We call the quadruple
		$(\mathbf a,\widetilde{\mathbf a},\mathbf b,\widetilde{\mathbf b})$
		admissible if $$\sum_{p \in \mathcal{P}} a_{j,p}=\sum_{p \in \mathcal{P}} \widetilde{a_{j,p}} = \ell_j, \qquad \sum_{p \in \mathcal{P}} b_{h,p} \leq X_h, \sum_{p \in \mathcal{P}} \widetilde{b_{h,p}}\leq X_h.$$ For an admissible sequence, we define $$A_p = \sum_{1\leq j \leq r} a_{j,p}+\sum_{1\leq h \leq q} b_{h,p} \qquad \widetilde{A_p} = \sum_{1 \leq j \leq r} \widetilde{a_{j,p}} + \sum_{1 \leq h \leq q} \widetilde{b_{h,p}}.$$ Then, if $\mathcal{Q}_K(F)$ denotes the quadratic form $$\mathcal{Q}_K(F)=\sum_{\boldsymbol{\alpha},\boldsymbol{\beta}} c_{\boldsymbol{\alpha}}\overline{c_{\boldsymbol{\beta}}}K(\boldsymbol{\alpha},\boldsymbol{\beta}),$$ we have
		\begin{align*}
			\mathcal{Q}_K(F)&={\prod_{j=1}^r (\ell_j!)^2 \sum_{\substack{(\mathbf{a},\widetilde{\mathbf{a}},\mathbf{b},\widetilde{\mathbf{b}})\\ \textup{admissible}}} \prod_{p \in \mathcal{P}} \left\{K_p(A_p,\widetilde{A_p})  \left(\prod_{j=1}^r \frac{u_j(p)^{a_{j,p}}\overline{u_j(p)}^{\widetilde{a_{j,p}}}}{a_{j,p}!\widetilde{a_{j,p}}!}  \prod_{h=1}^q \frac{v_h(p)^{b_{h,p}}\overline{v_h(p)}^{\widetilde{b_{h,p}}}}{b_{h,p}!\widetilde{b_{h,p}}!}\right)\right\}.}
		\end{align*}
	\end{lem}

	\begin{proof}
		Expand each $L_j^{\ell_j}$ and $E_{X_h}(M_h)$ by the multinomial
		theorem.  Their coefficients are respectively
		\[
		\ell_j!\prod_p\frac{u_j(p)^{a_{j,p}}}{a_{j,p}!},
		\qquad
		\prod_p\frac{v_h(p)^{b_{h,p}}}{b_{h,p}!},
		\]
		with the stated degree restrictions.
		Applying the same expansion to the conjugate copy of $F$ and inserting
		$K=\prod_pK_p$ gives the formula.
	\end{proof}
	
	{The balanced specialization $K_p(a,b)=a!\mathbf{1}_{a=b}$ gives the following result.}
	
	\begin{lem}\label{balancedCoefficientExpansionLemma}
		Let $\mathcal{P}, \mathbf{z}, L_j(\mathbf{z}), M_h(\mathbf{z}),(X_1,\dots,X_q), (\ell_1,\dots,\ell_r)$ be as in the previous lemma. For ease of notation, write $${\mathcal{Y}_j(\mathbf z)}=\sum_{p\in \mathcal{P}} y_j(p)z_p,$$ where $\mathcal{Y}_j=L_j$ for $1\leq j \leq r$; and $\mathcal{Y}_{r+h}=M_h$ for $1\leq h \leq q$. Define $C_{a,b}=\sum_{p \in \mathcal{P}} y_{a}(p)\overline{y_{b}(p)}$. Then, if $$F(\mathbf{z})=\prod_{j=1}^r L_j^{\ell_j}(\mathbf{z})\prod_{h=1}^q E_{X_h}(M_h(\mathbf{z}))=\sum_{\boldsymbol{\alpha}} c_{\boldsymbol{\alpha}}\mathbf{z}^{\boldsymbol{\alpha}},$$ we have $$\sum_{\boldsymbol{\alpha}} \boldsymbol{\alpha}!|c_{\boldsymbol{\alpha}}|^2=\prod_{j=1}^r (\ell_j!)^2\sum_{\rho \in \mathscr{R}}\prod_{a,b=1}^{r+q} \frac{C_{a,b}^{\rho_{a,b}}}{\rho_{a,b}!},$$ where $\boldsymbol{\alpha}!=\prod_{p\in \mathcal{P}} \alpha_p!$ and {$\mathscr{R}$ is the set of non-negative integer matrices $\rho$ satisfying
			\[
			\sum_{i=1}^{r+q}\rho_{j,i}=
			\sum_{i=1}^{r+q}\rho_{i,j}=\ell_j
			\quad(1\leq j\leq r),
			\]
			and
			\[
			\sum_{i=1}^{r+q}\rho_{r+h,i}\leq X_h,
			\qquad
			\sum_{i=1}^{r+q}\rho_{i,r+h}\leq X_h
			\quad(1\leq h\leq q).
			\]}
	\end{lem}
	
	\begin{proof}
		Apply \Cref{coefficientKernelBookkeepingLemma} with
		$K_p(A,B)=A!\mathbf1_{A=B}$.  For non-negative integers $m_i,\widetilde m_j$
		with common total $A$, the multinomial theorem gives
		\begin{equation}\label{balancedTransportIdentity}
			\frac{A!}{\prod_i m_i!\prod_j\widetilde m_j!}
			=
			\sum_{\substack{\rho_{ij}\geq0\\
					\sum_j\rho_{ij}=m_i,\ \sum_i\rho_{ij}=\widetilde m_j}}
			\frac1{\prod_{i,j}\rho_{ij}!}.
		\end{equation}
		At each prime, this expresses the local contribution as
		$\prod_{i,j}\frac{(y_i(p)\overline{y_j(p)})^{\rho^{(p)}_{ij}}}
		{\rho^{(p)}_{ij}!}$.  For $\rho_{ij}=\sum_p\rho^{(p)}_{ij}$,
		another multinomial expansion gives
		\[
		\sum_{\substack{\rho^{(p)}_{ij}\geq0\\
				\sum_p\rho^{(p)}_{ij}=\rho_{ij}}}
		\prod_p\frac{(y_i(p)\overline{y_j(p)})^{\rho^{(p)}_{ij}}}
		{\rho^{(p)}_{ij}!}
		=\frac{C_{ij}^{\rho_{ij}}}{\rho_{ij}!}.
		\]
		The global row and column sums are the total degrees in the two
		copies of $F$, hence satisfy exactly the restrictions defining
		$\mathscr R$.
	\end{proof}
	
	\begin{lem}\label{unbalancedTransportLemma}
		Let $\mathcal{I}$ be a finite index set. For each $i\in \mathcal{I}$, we are given non-negative integers $m_i, \widetilde{m_i}$ and complex numbers $y_i$. Let $A=\sum_{i \in \mathcal{I}} m_i$ and $\widetilde{A}=\sum_{i \in \mathcal{I}} \widetilde{m_i}$. 
		
		Suppose $A>\widetilde{A}$ and let $d=A-\widetilde{A}$. Then $$\frac{\prod_{i \in \mathcal{I}} y_i^{m_i}\overline{y_i}^{\widetilde{m_i}}}{\prod_{i \in \mathcal{I}} m_i!\widetilde{m_i}!}=\frac{d!}{A!}\sum_{\substack{(\eta_i)_{i \in \mathcal{I}}, \\ 0\leq \eta_i\leq m_i, \\ \sum_{i \in \mathcal{I}} \eta_i = d}} \sum_{\substack{(\rho_{i,j})_{i,j \in \mathcal{I}} \geq 0, \\ \sum_{j\in \mathcal{I}} \rho_{i,j}=m_i-\eta_i, \\ \sum_{i \in \mathcal{I}} \rho_{i,j} = \widetilde{m_j} }} \left(\prod_{i \in \mathcal{I}} \frac{y_i^{\eta_i}}{\eta_i!}\right) \prod_{i,j \in \mathcal{I}} \frac{(y_i\overline{y_j})^{\rho_{i,j}}}{\rho_{i,j}!}.$$
		
		In the case $\widetilde{A}>A$ the same conclusion holds {after interchanging the two sides}: $d=\widetilde{A}-A$; and $$\frac{\prod_{i \in \mathcal{I}} y_i^{m_i}\overline{y_i}^{\widetilde{m_i}}}{\prod_{i \in \mathcal{I}} m_i!\widetilde{m_i}!} = \frac{d!}{\widetilde{A}!} \sum_{\substack{(\widetilde{\eta_j})_{j \in \mathcal{I}}, \\ 0\leq \widetilde{\eta_j}\leq \widetilde{m_j}, \\ \sum_{j \in \mathcal{I}} \widetilde{\eta_j} = d}} \sum_{\substack{(\rho_{i,j})_{i,j \in \mathcal{I}} \geq 0, \\ \sum_{j\in \mathcal{I}} \rho_{i,j}=m_i, \\ \sum_{i \in \mathcal{I}} \rho_{i,j} = \widetilde{m_j} - \widetilde{\eta_j} }} \left(\prod_{j \in \mathcal{I}} \frac{\overline{y_j}^{\widetilde{\eta_j}}}{\widetilde{\eta_j}!}\right) \prod_{i,j \in \mathcal{I}} \frac{(y_i\overline{y_j})^{\rho_{i,j}}}{\rho_{i,j}!}.$$
	\end{lem}
	
	\begin{proof}
		Suppose $A>\widetilde A$.  For each admissible $\boldsymbol\eta$,
		apply \eqref{balancedTransportIdentity} to the row sums
		$m_i-\eta_i$ and column sums $\widetilde m_j$.
		Vandermonde's identity gives
		\[
		\sum_{\substack{0\leq\eta_i\leq m_i\\\sum_i\eta_i=d}}
		\frac1{\prod_i\eta_i!(m_i-\eta_i)!}
		=\frac1{\prod_i m_i!}\binom Ad
		=\frac{A!}{d!\widetilde A!\prod_i m_i!}.
		\]
		The double sum of reciprocal factorials is therefore
		\[
		\frac{\widetilde A!}{\prod_j\widetilde m_j!}
		\frac{A!}{d!\widetilde A!\prod_i m_i!}
		=\frac{A!}{d!\prod_i m_i!\widetilde m_i!}.
		\]
		Each admissible term has the same monomial, since
		$m_i=\eta_i+\sum_j\rho_{ij}$ and
		$\widetilde m_j=\sum_i\rho_{ij}$.
		Multiplication by $d!/A!$ proves the identity.
		Interchanging the two sides proves the other case.
	\end{proof}
	
	{\section{Proofs of the Dirichlet-polynomial kernel estimates}\label{appendixDirichletPolynomialKernels}}
	{We prove the kernel estimates from
		Section~\ref{twistedMomentsSection}.}
	{\subsection{The untwisted product estimate}}
	\begin{proof}[Proof of \Cref{highPowersAndSomeSquaresUntwistedLemma}]
		The length assumption lets us apply Montgomery--Vaughan, so upon utilizing the trivial bound $$\sum_{\boldsymbol{\alpha}} |c_{\boldsymbol{\alpha}}|^2 \leq \sum_{\boldsymbol{\alpha}} \boldsymbol{\alpha}!|c_{\boldsymbol{\alpha}}|^2,$$ \Cref{balancedCoefficientExpansionLemma} gives {immediately} that $$\int_{T}^{2T} \prod_{j=1}^\mu \left|\mathcal{P}_j(\tfrac12+it)\right|^{2\ell_j} \prod_{j=1}^{\nu} |\mathcal{N}_j(\tfrac12+it)|^2 \dt \ll T \prod_{j=1}^{\mu} (\ell_j!)^2\sum_{\rho \in \mathscr{R}} \prod_{a,b=1}^{\mu+\nu} \frac{V_{a,b}^{\rho_{a,b}}}{\rho_{a,b}!}.$$ 
		
		Fix a transport matrix $\rho$ in the sum above. For $1\leq j \leq \mu$ let $m_j=\ell_j - \rho_{jj}$. By the admissibility condition produced in \Cref{balancedCoefficientExpansionLemma}, we have
		\begin{equation}\label{mjAdmissibilityCondition}
			m_j = \sum_{1\leq i \leq r, i \neq j} \rho_{i,j} = \sum_{1\leq i \leq r, i \neq j} \rho_{j,i}.
		\end{equation}
		
		By the definition of $m_j$, we have
		\begin{align*}
			\frac{\prod_{i=1}^{\mu} (\ell_i!)^2}{\prod_{i,j} \rho_{i,j}!} \prod_{i,j} V_{i,j}^{\rho_{i,j}} &= \left(\prod_{i=1}^\mu \ell_i!\right) \left(\prod_{i=1}^\mu \frac{\ell_i!V_i^{\rho_{i,i}}}{\rho_{ii}!} \right)\left( \prod_{i=\mu+1}^r \frac{V_i^{\rho_{i,i}}}{\rho_{ii}!}\right) \left(\prod_{1\leq i,j\leq r, i\neq j} \frac{V_{i,j}^{\rho_{i,j}}}{\rho_{i,j}!}\right) \\
			&=\left(\prod_{i=1}^\mu \ell_i!\right) \left(\prod_{i=1}^{\mu} V_{i}^{\ell_i} \frac{\ell_i!}{(\ell_i-m_i)!V_i^{m_i}}\right)\left( \prod_{i=\mu+1}^r \frac{V_i^{\rho_{i,i}}}{\rho_{ii}!}\right) \left(\prod_{1\leq i,j\leq r, i\neq j} \frac{V_{i,j}^{\rho_{i,j}}}{\rho_{i,j}!}\right) \\
			&\ll  \left(\prod_{i=1}^\mu \ell_i!V_i^{\ell_i}\right) \left(\prod_{i=1}^{\mu} \left(\frac{\ell_i}{V_i}\right)^{m_i} \right) \left( \prod_{i=\mu+1}^r \frac{V_i^{\rho_{i,i}}}{\rho_{ii}!}\right) \left(\prod_{1\leq i,j\leq r, i\neq j} \frac{|V_{i,j}|^{\rho_{i,j}}}{\rho_{i,j}!}\right),
		\end{align*}
		where we used the trivial bound $\frac{(\ell_i)!}{(\ell_i-m_i)!} \leq \ell_i^{m_i}$. 
		
		By (\ref{mjAdmissibilityCondition}) we can rewrite one of the factors as follows $$\prod_{i=1}^{\mu} \left(\frac{\ell_i}{V_i}\right)^{m_i} = \prod_{i=1}^{\mu} \left(\frac{\ell_i}{V_i}\right)^{\sum_{1\leq j \leq r, i \neq j} \rho_{i,j}}=\prod_{i=1}^{\mu} \prod_{1\leq j \leq r, i\neq j} \left(\frac{\ell_i}{V_i}\right)^{\rho_{i,j}}.$$ Thus
		\begin{align*}
			&\prod_{i=1}^{\mu} \left(\frac{\ell_i}{V_i}\right)^{m_i}
			\left(\prod_{\substack{1\leq i,j\leq r\\i\neq j}}
			\frac{|V_{i,j}|^{\rho_{i,j}}}{\rho_{i,j}!}\right)\\
			&\quad=\left(\prod_{1\leq i \leq \mu}
			\prod_{\substack{1\leq j \leq r\\j\neq i}}
			\left(\frac{\ell_i}{V_i}\right)^{\rho_{i,j}}
			\frac{|V_{i,j}|^{\rho_{i,j}}}{\rho_{i,j}!}\right)
			\left(\prod_{\substack{\mu+1\leq i \leq r,\\ 1\leq j \leq r, j\neq i}}\frac{|V_{i,j}|^{\rho_{i,j}}}{\rho_{i,j}!}\right) \\
			&= \left(\prod_{1\leq i \leq \mu} \prod_{1\leq j \leq r, i \neq j} \frac{1}{\rho_{i,j}!}\left(\frac{\ell_i|V_{i,j}|}{V_i}\right)^{\rho_{i,j}} \right)\left(\prod_{\substack{\mu+1\leq i \leq r,\\ 1\leq j \leq r, j\neq i}}\frac{|V_{i,j}|^{\rho_{i,j}}}{\rho_{i,j}!}\right).
		\end{align*}
		
		With this algebraic manipulation and upper bound in place, we are ready to sum $$\frac{\prod_{i=1}^{\mu} (\ell_i!)^2}{\prod_{i,j} \rho_{i,j}!} \prod_{i,j} V_{i,j}^{\rho_{i,j}}$$ over all admissible matrices $\rho$. Observe that {the resulting product} has a form where the factors are formed as products of terms coming from the Taylor series of $\exp$. Indeed, for $1\leq i\leq \mu$, the diagonal entry $\rho_{ii}$ is determined by the off-diagonal entries through
		$$
		\rho_{ii} = \ell_i-\sum_{\substack{1\leq j\leq r\\ j\neq i}}\rho_{i,j}.$$
		After taking absolute values, all terms in our upper bound are non-negative. We may therefore drop the remaining admissibility conditions and extend the ranges of all off-diagonal entries $\rho_{i,j}$, $i\neq j$, and the diagonal entries $\rho_{i,i}$, $i>\mu$, independently to all non-negative integers. The resulting sums separate and are Taylor series for the exponential.
		$$\sum_{\rho \in \mathscr{R}}\left|\frac{\prod_{i=1}^{\mu} (\ell_i!)^2}{\prod_{i,j} \rho_{i,j}!} \prod_{i,j} V_{i,j}^{\rho_{i,j}}\right| \ll \left(\prod_{i=1}^{\mu} \ell_i! V_i^{\ell_i}\right)\exp\left(\sum_{i=\mu+1}^r V_i + \sum_{\substack{1\leq i \leq \mu, \\ 1 \leq j \leq r, i \neq j}} \frac{\ell_i}{V_i}|{V_{i,j}}| + \sum_{\substack{\mu+1 \leq i \leq r, \\ 1 \leq j \leq r, i \neq j}} |{V_{i,j}}|\right)$$
		Using that $\ell_{i} \ll V_{i}$ and ${V_{i,j}}=O(1)$ for $i\neq j$, we conclude that $$\sum_{\rho\in \mathscr{R}}\frac{\prod_{i=1}^{\mu} (\ell_i!)^2}{\prod_{i,j} \rho_{i,j}!} \prod_{i,j} V_{i,j}^{\rho_{i,j}} \ll_r \left(\prod_{j=1}^\mu \ell_j! V_j^{\ell_j}\right)\left( \prod_{j=1}^{\nu} \exp\left( V_{j+\mu}\right)\right),$$ which is what we wanted to prove.
	\end{proof}
	
	{\subsection{Verification of the local kernels}}
	\begin{proof}[Proof of \Cref{verificationOfTwistedLocalKernels}]
		{At $p$, the factor $\sqrt{nm}/[n,m]$ is
			$p^{-|A-B|/2}$.  Combining this identity with the definitions of
			$B,B^\chi$ proves \eqref{zetaFourthLocalKernel} and
			\eqref{characterFourthLocalKernel}.  Combining it with
			\eqref{normalizedCuspLocalFactor} proves
			\eqref{cuspSecondLocalKernel}.  These
			formulae also show immediately that the kernels depend only on $A-B$.
			
			Expanding the two geometric series in the definition of $B$ gives,
			uniformly for the present shifts,
			\[
			B_{\mathbf z,p^d},\ B^\chi_{\mathbf z,p^d}
			\ll(d+1)C^d,
			\]
			and keeping the terms of total local degree one gives
			\eqref{zetaLinearLocalTerms}--\eqref{characterLinearLocalTerms}.
			Equations \eqref{cuspPlusKappa} and
			\eqref{cuspMinusKappa} follow directly from the displayed Euler
			factors.  Deligne's bound (and the Deligne--Serre bound when
			$\kappa=1$) $|\lambda_f(p^d)|\leq d+1$, together with
			$|p^{O(1/\log T)}|\leq e^{O(c)}$, proves
			\eqref{verifiedKappaBound} in all three cases.
			
			For a twist by $L_i$, consecutive prime blocks, and coefficients
			$b_a(p)=\alpha_a\lambda_a(p)/2$ with $\alpha_a\in\mathbb R$,
			partial summation gives
			\begin{align}
				\left|\sum_hW_{a,h}^+\right|
				+\left|\sum_hW_{a,h}^-\right|
				&\ll1 &&(a\neq i),\label{verifiedNondistinguishedW}\\
				\Re\sum_h(W_{i,h}^++W_{i,h}^-)
				&=\frac{u\alpha_i}{2}
				\sum_{p\ \mathrm{in\ the\ blocks}}
				\frac{|\lambda_i(p)|^2}{p}+O(1),
				\label{verifiedDistinguishedW}
			\end{align}
			where $u=4$ for a fourth-moment kernel and $u=2$ for the cusp-form
			second-moment kernel.  Indeed replacing a factor $p^{-z_j}$ or
			$p^{-w/2}$ by $1$ changes the complete prime sum by $O(1)$, uniformly
			for $|z_j|\ll1/\log T$; this is another direct partial-summation
			consequence of
			\eqref{rankinSelbergQuantitativePrimeTheorem} and
			\eqref{rankinSelbergQuantitativeMertens}.}
	\end{proof}
	
	{\subsection{The single-block and multiblock kernel estimates}}
	{We next derive a twisted analogue of
		\Cref{highPowersAndSomeSquaresUntwistedLemma}.  The formulation
		applies whenever the available twisted-moment kernel satisfies the
		stated Ramanujan-type coefficient bounds.}
	
	{The following lemma separates the non-summable local configurations from the absolutely convergent remainder and reorganizes the former globally.}
	
	\begin{lem}\label{twistedCombinatorialKernelLemma}
		Let $\mathcal{P}, r, \mu, \nu, \ell_j, X_j, \mathcal{N}_j, V_{i,j}$ and $V_j$ be as defined in \Cref{highPowersAndSomeSquaresUntwistedLemma}. Let $$\mathcal{A}(s)=\sum_{n} \frac{a(n)}{n^s} =\prod_{j=1}^\mu (\mathcal{P}_j(s))^{\ell_j} \prod_{j=1}^{\nu} \mathcal{N}_j(s)$$ 
		Assume that the coefficient $b_j(p)$ in the definition of $\mathcal{P}_j(s)$ satisfies $$b_j(p)\ll p^{\theta}$$ for some $0\leq \theta < 1/4$.
		Assume that ${V_{i,j}}=O(1)$ for $i\neq j$, and that $\ell_{i} \ll V_i$ for $1\leq i \leq \mu$. Let $K_p:\mathbb{Z}_{\geq 0} \times \mathbb{Z}_{\geq 0} \to \mathbb{C}$ be a kernel defined for all $p \in \mathcal{P}$ that satisfies:
		\begin{enumerate}
			\item $K_p(a,a)=1$ for $a\geq 0$ for all $p\in \mathcal{P}$.
			\item There exist quantities $\kappa_{p,d}^+$ and $\kappa_{p,d}^-$ independent of $a$ that satisfy $$K_p(a+d,a)=\kappa_{p,d}^+, \qquad K_p(a,a+d)=\kappa_{p,d}^-$$ when $d\geq 1$.
			\item There exists an absolute constant $C>0$ such that the parameters $\kappa$ satisfy $$|\kappa^{+}_{p,d}|+|\kappa^{-}_{p,d}|\ll (d+1)C^d p^{-d/2+\theta}$$
		\end{enumerate}
		We define the global kernel $K$ by $$K(n,m)=\prod_{p \in \mathcal{P}} K_p(v_p(n),v_p(m)),$$ where $v_p(n)$ is the $p$-adic valuation of $n$. 
		
		To the $\kappa$-parameters we associate a family of weighted sums $$W_j^+= \sum_{p \in \mathcal{P}} \frac{b_j(p)}{\sqrt{p}}\kappa^+_{p,1}, \qquad W_j^- = \sum_{p \in \mathcal{P}} \frac{\overline{b_j(p)}}{\sqrt{p}}\kappa^-_{p,1},$$ for which we assume that, apart from a distinguished index $j_0$, $$W_j^+ = O(1), W_j^- = O(1) \textup{ for } j\neq j_0. $$ The $j_0$-terms are not assumed to be bounded. Finally, we assume that the cut-off parameters $X_j$ satisfy $$X_j \geq C_1(V_{\mu+j}+|W^+_{\mu+j}| + |W^-_{\mu+j}|+1)$$ for $1\leq j \leq \nu$ for some sufficiently large constant $C_1$. 
		Then the following bounds hold.
		
		\begin{enumerate}
			\item If $j_0>\mu$ we have $$\left|\sum_{n,m} \frac{a(n)\overline{a(m)}}{\sqrt{nm}}K(n,m)\right| \ll \left(\prod_{i=1}^{\mu} \ell_i! V_i^{\ell_i}\right)\exp \left(\sum_{h=1}^{\nu} V_{\mu+h}\right)\exp(\textup{Re}\left(W_{j_0}^- + W_{j_0}^+\right)).$$
			\item If $j_0 \leq \mu$, assume in addition that $|W^+_{j_0}|\asymp V_{j_0}$ and $|W^{-}_{j_0}|\asymp V_{j_0}$. Then we have $$\left|\sum_{n,m} \frac{a(n)\overline{a(m)}}{\sqrt{nm}}K(n,m)\right| \ll \left(\prod_{i=1, i \neq j_0}^{\mu} \ell_i! V_i^{\ell_i}\right)\exp \left(\sum_{h=1}^{\nu} V_{\mu+h}\right) \mathcal{H}_{\ell_{j_0}}(V_{j_0}; W_{j_0}^+, W_{j_0}^-),$$ where $$\mathcal{H}_{\ell}(V;W^+,W^-)=(\ell!)^2\sum_{j=0}^{\ell} \frac{V^{\ell-j}|W^+W^-|^j}{(\ell-j)!(j!)^2}.$$
		\end{enumerate}
	\end{lem}
	
	\begin{rem}
		We note the bound for $\mathcal{H}_{\ell}(V;W^+,W^-)$. Indeed, using the inequalities $\frac{\ell!}{(\ell-j)!}\leq \ell^j$ and  $$\sum_{j\geq 0} \frac{x^j}{(j!)^2}\leq \sum_{j\geq 0} \frac{4^jx^j}{(2j)!} \leq \sum_{j \geq 0} \frac{(2\sqrt{x})^{2j}}{(2j)!}\leq e^{2\sqrt{x}}$$ for $x\geq 0$, we get that $$(\ell!)^2\sum_{j=0}^{\ell} \frac{V^{\ell-j}|W^+W^-|^j}{(\ell-j)!(j!)^2}\leq \ell! V^{\ell} \sum_{j=0}^{\ell} \frac{1}{(j!)^2}\left(\ell\frac{|W^+W^-|}{V}\right)^j \leq \ell! V^{\ell}\exp\left(2 \sqrt{\ell\frac{|W^+W^-|}{V}}\right).$$
	\end{rem}
	
	\begin{proof}
		{The proof follows
			\Cref{highPowersAndSomeSquaresUntwistedLemma}, with the additional
			treatment of the unbalanced configurations
			$A_p\ne\widetilde A_p$.}  As before, for each $p\in \mathcal{P}$, we have sequences of non-negative integers $m_{1,p},\dots,m_{r,p}$ and $\widetilde{m_{1,p}},\dots,\widetilde{m_{r,p}}$. We say that $((m_{j,p}), (\widetilde{m_{j,p}}))$ forms an admissible pair if
		\begin{enumerate}
			\item $\sum_{p \in \mathcal{P}} m_{j,p}=\sum_{p \in \mathcal{P}} \widetilde{m_{j,p}} = \ell_j \text{ for } 1\leq j \leq \mu$ and 
			\item $\sum_{p \in \mathcal{P}} m_{\mu+h,p}, \sum_{p \in \mathcal{P}} \widetilde{m_{\mu+h,p}} \leq X_h \text{ for } 1 \leq h \leq \nu$
		\end{enumerate}
		{As above, put} $A_p = \sum_{j} m_{j,p}$ and $\widetilde{A_p} = \sum_j \widetilde{m_{j,p}}$. Applying \Cref{coefficientKernelBookkeepingLemma} with $u_j(p)=b_j(p)p^{-1/2}$ for $1\leq j \leq \mu$ and $v_j(p)=b_{\mu+j}(p)p^{-1/2}$ for $1\leq j \leq \nu$, we get 
		$$\sum_{n,m} \frac{a(n)\overline{a(m)}}{\sqrt{nm}}K(n,m) = \prod_{j=1}^{\mu} (\ell_j!)^2 \sum_{\substack{((m_{j,p}), (\widetilde{m_{j,p}})) \\ \textup{admissible}}} \prod_{p \in \mathcal{P}} K_p(A_p,\widetilde{A_p})\left(\prod_{j=1}^r \frac{b_j(p)^{m_{j,p}}\overline{b_j(p)}^{\widetilde{m_{j,p}}}}{p^{(m_{j,p}+\widetilde{m_{j,p}})/2}m_{j,p}!\widetilde{m_{j,p}}!}\right).$$
		We estimate the local contributions.
		
		Put
		\[
		B_p=\sum_{i=1}^r\frac{|b_i(p)|}{\sqrt p}\ll p^{-1/2+\theta},
		\qquad
		D_p=\sum_{i,j=1}^r\frac{|b_i(p)b_j(p)|}{p}
		=B_p^2\ll p^{-1+2\theta}.
		\]
		The configurations $(A_p,\widetilde A_p)=(1,1),(1,0),(0,1)$ contribute
		respectively
		\[
		\frac{b_i(p)\overline{b_j(p)}}p,\qquad
		\kappa_{p,1}^+\frac{b_i(p)}{\sqrt p},\qquad
		\kappa_{p,1}^-\frac{\overline{b_i(p)}}{\sqrt p}.
		\]
		Their prime sums are $V_{i,j}$ and $W_i^\pm$.
		Apart from the constant term, all remaining configurations are absolutely summable:
		
		If $A_p=\widetilde A_p=A\geq2$, \eqref{balancedTransportIdentity}
		and the multinomial theorem give the total bound
		\[
		\sum_{A\geq2}\frac{D_p^A}{(A!)^2}\ll p^{-2+4\theta}.
		\]
		If $|A_p-\widetilde A_p|=1$ with both degrees positive,
		\Cref{unbalancedTransportLemma} leaves one unpaired occurrence and
		a non-empty transport matrix.  Summing both gives
		\[
		(|\kappa_{p,1}^+|+|\kappa_{p,1}^-|)B_pD_pe^{D_p}
		\ll p^{-2+4\theta}.
		\]
		Finally, for $d=|A_p-\widetilde A_p|\geq2$, the unpaired degrees sum
		to $B_p^d/d!$, the transport matrices to at most $e^{D_p}$, and
		$d!/\max(A_p,\widetilde A_p)!\leq1$.  Thus
		\[
		e^{D_p}\sum_{d\geq2}
		\frac{|\kappa_{p,d}^+|+|\kappa_{p,d}^-|}{d!}B_p^d
		\ll\sum_{d\geq2}\frac{(d+1)(C')^d}{d!}p^{-d+(d+1)\theta}
		\ll p^{-2+3\theta}.
		\]
		Both exponents are less than $-1$ because $\theta<1/4$.
		{
			We now make the passage from the local estimates to the global
			generating series exact.  Introduce independent degree variables
			$\mathbf z=(z_1,\ldots,z_r)$ and
			$\mathbf w=(w_1,\ldots,w_r)$, and put
			\[
			\mathfrak F_p(\mathbf z,\mathbf w)
			=\sum_{\mathbf m,\widetilde{\mathbf m}\geq0}
			K_p(A,\widetilde A)
			\prod_{i=1}^r
			\frac{(b_i(p)z_i/\sqrt p)^{m_i}}{m_i!}
			\frac{(\overline{b_i(p)}w_i/\sqrt p)^{\widetilde m_i}}
			{\widetilde m_i!},
			\]
			where $A=\sum_i m_i$ and $\widetilde A=\sum_i\widetilde m_i$.
			The three non-absolutely-summable local types found above form
			\[
			\mathfrak L_p(\mathbf z,\mathbf w)
			=\sum_{i,j=1}^r\frac{b_i(p)\overline{b_j(p)}}p z_iw_j
			+\sum_{i=1}^r\kappa_{p,1}^+\frac{b_i(p)}{\sqrt p}z_i
			+\sum_{i=1}^r\kappa_{p,1}^-\frac{\overline{b_i(p)}}{\sqrt p}w_i.
			\]
			Thus, exactly,
			$\mathfrak F_p=1+\mathfrak L_p+\mathfrak E_p$.
			
			For a power series
			$G=\sum_{\mathbf a,\mathbf b}g_{\mathbf a,\mathbf b}
			\mathbf z^{\mathbf a}\mathbf w^{\mathbf b}$ and a fixed $R>1$, write
			\[
			\|G\|_R=\sum_{\mathbf a,\mathbf b}|g_{\mathbf a,\mathbf b}|
			R^{|\mathbf a|+|\mathbf b|}.
			\]
			The preceding balanced and unbalanced estimates remain valid after
			weighting every configuration by $R$ to its total labelled degree.
			They give
			\begin{equation}\label{weightedLocalCompletionRemainder}
				\|\mathfrak E_p\|_R+\|\mathfrak L_p\|_R^2
				\ll_Rp^{-1-\delta},
				\qquad \delta=1-4\theta>0.
			\end{equation}
			Since $\|GH\|_R\leq\|G\|_R\|H\|_R$, the product
			\[
			\mathfrak H(\mathbf z,\mathbf w)
			:=\prod_{p\in\mathcal P}\left(
			\mathfrak F_p(\mathbf z,\mathbf w)
			\exp(-\mathfrak L_p(\mathbf z,\mathbf w))\right)
			\]
			has absolutely summable weighted coefficients and, for every fixed $R>1$,
			\begin{equation}\label{singleBlockAnalyticRemainderBound}
				\|\mathfrak H\|_R\ll_R1.
			\end{equation}
			Indeed,
			$\|\mathfrak F_pe^{-\mathfrak L_p}-1\|_R
			\ll_R\|\mathfrak E_p\|_R+\|\mathfrak L_p\|_R^2$.
			We have proved the exact factorization
			\begin{equation}\label{singleBlockExactCompletion}
				\prod_{p\in\mathcal P}\mathfrak F_p(\mathbf z,\mathbf w)
				=\mathfrak H(\mathbf z,\mathbf w)
				\exp\left(
				\sum_{i,j=1}^rV_{i,j}z_iw_j
				+\sum_{i=1}^rW_i^+z_i
				+\sum_{i=1}^rW_i^-w_i
				\right).
			\end{equation}
			By \Cref{coefficientKernelBookkeepingLemma}, the quadratic form is
			$\prod_{i\leq\mu}(\ell_i!)^2$ times the coefficient sum with degree
			$(\ell_i,\ell_i)$ in every fixed pair and degrees at most $(X_h,X_h)$
			in every truncated pair.
			
			First complete the truncated factors by setting their variables to $1$.
			The terms with $V_{i,j}$, $i\ne j$, and the nondistinguished $W_i^\pm$
			can be included in $\mathfrak H$: their total absolute coefficient sum
			is bounded, so \eqref{singleBlockAnalyticRemainderBound} still holds.
			For an ordinary fixed pair and $0\leq a,b\leq\ell$, a term of this
			remainder that removes degrees $a,b$ satisfies
			\begin{equation}\label{ordinaryFixedCoefficientShift}
				\frac{(\ell!)^2[z^{\ell-a}w^{\ell-b}]e^{Vzw}}{\ell!V^\ell}
				=\mathbf1_{a=b\leq\ell}\frac{\ell!}{(\ell-a)!V^a}
				\leq\mathbf1_{a=b\leq\ell}(\ell/V)^a.
			\end{equation}
			For a distinguished fixed pair, put $u=|W^+|$, $t=|W^-|$, and
			$C_{m,n}=[z^mw^n]e^{Vzw+uz+tw}$. Termwise comparison gives
			\begin{equation}\label{distinguishedFixedCoefficientShift}
				C_{\ell-a,\ell-b}
				\leq(\ell/u)^a(\ell/t)^b C_{\ell,\ell}
				\qquad(0\leq a,b\leq\ell).
			\end{equation}
			Indeed, supply the missing powers from $uz$ and $tw$; the factorial
			ratios are at most $\ell^a$ and $\ell^b$. Since $\ell\ll V$ and
			$u,t\asymp V$, a fixed radius bounds every degree change, including
			$a\ne b$. Summing against the weighted absolute coefficients of
			$\mathfrak H$ therefore costs a fixed factor.
			
			Each ordinary fixed pair contributes $\ell!V^\ell$, and each
			completed truncated pair contributes $e^V$. If $j_0>\mu$, the
			distinguished linear terms give
			$\exp\Re(W_{j_0}^++W_{j_0}^-)$. If $j_0\leq\mu$, they give
			\[
			(\ell!)^2C_{\ell,\ell}=\mathcal H_\ell(V;W^+,W^-).
			\]
			Thus the completed expression satisfies the respective asserted bound;
			write $\mathcal M$ for that bound.
			
			To restore all truncations, let $E$ be any subset of their $z$- and
			$w$-variables, and let $\mathcal Q_E$ be the contribution where each
			selected degree exceeds its cutoff. Leave the other truncated variables
			at $1$. Cauchy's formula on circles of fixed radius $R>2$ has tail
			kernel bounded by $R^{-X_h}/(R-1)$. The preceding coefficient
			comparisons hold uniformly on these circles, and give
			\[
			|\mathcal Q_E|\leq C\mathcal M
			\prod_{(h,\pm)\in E}\exp(-X_h\log R+C_RS_h),\qquad
			S_h=V_{\mu+h}+|W_{\mu+h}^+|+|W_{\mu+h}^-|+1.
			\]
			Here $+$ and $-$ indicate the $z$- and $w$-degree respectively.
			Choose $C_1\log R\geq2C_R$. Since $X_h\geq C_1S_h$,
			inclusion--exclusion bounds the truncated expression by
			\[
			C\mathcal M\prod_{h=1}^{\nu}(1+e^{-X_h\log R/2})^2
			\ll\mathcal M,
			\]
			because $\nu$ is fixed. This restores both degrees of every truncated
			factor, preserving the signed distinguished linear terms, and proves
			both assertions.
		}
	\end{proof}
	
	\begin{proof}[Proof of
		\Cref{multiBlockCombinatorialLemma}]
		{Give each label $\gamma=(a,h)$ its own variables
			$z_\gamma,w_\gamma$.  Define $\mathfrak F_p,\mathfrak L_p,\mathfrak E_p$
			as in the single-block proof, using only the labels on the block containing
			$p$, and the same weighted coefficient sum $\|\cdot\|_R$.
			The estimate \eqref{weightedLocalCompletionRemainder} is uniform in
			$h$, since each prime belongs to one block with at most $r_0$ labels.
			Since $\sum_p p^{-2+4\theta}<\infty$, we obtain one global bound
			\[
			\mathfrak H=\prod_p(\mathfrak F_pe^{-\mathfrak L_p}),
			\qquad \|\mathfrak H\|_R\ll_R1,
			\]
			with a constant independent of $H$, and exactly
			\begin{equation}\label{multiblockExactEulerCompletion}
				\prod_p\mathfrak F_p
				=\mathfrak H\exp\left(
				\sum_{\gamma,\delta}V_{\gamma,\delta}z_\gamma w_\delta
				+\sum_\gamma W_\gamma^+z_\gamma
				+\sum_\gamma W_\gamma^-w_\gamma\right).
			\end{equation}
			
			Complete every truncated family by setting
			$z_\gamma=w_\gamma=1$ for $\gamma\in\mathfrak T$.  The
			truncated--truncated quadratic exponent becomes
			\[
			\sum_{\gamma\in\mathfrak T}V_\gamma
			+\sum_{a\ne b}\sum_{\substack{h:(a,h),(b,h)\in\mathfrak T}}
			V_{(a,h),(b,h)}.
			\]
			Its off-diagonal part is $O(1)$ by
			\eqref{multiblockAggregateCovariance}.  Likewise
			\eqref{multiblockAggregateLinear} makes all completed
			non-distinguished linear terms $O(1)$ after the block sum, while the
			distinguished terms give
			$W_{\mathfrak T,0}^++W_{\mathfrak T,0}^-$.  Off-diagonal terms involving a fixed power have bounded sum of
			absolute coefficients by the second condition in
			\eqref{multiblockAggregateCovariance}, as do the remaining
			nondistinguished linear terms by \eqref{multiblockAggregateLinear}.
			
			Extract degree $(\ell_\gamma,\ell_\gamma)$ in each fixed pair.
			The comparisons \eqref{ordinaryFixedCoefficientShift} and
			\eqref{distinguishedFixedCoefficientShift} bound every change of these
			degrees when expanding the remainder and the bounded cross terms.
			Their weighted absolute coefficient sum is bounded independently of
			$H$, so there is a single fixed factor for all the fixed pairs.
			The resulting bound is the right-hand side of
			\eqref{multiblockTwistedConclusion}; denote it by $\mathcal M$.
			We restore the truncations simultaneously. For any set $E$ of truncated
			sides $(\gamma,+)$ or $(\gamma,-)$, let $\mathcal Q_E$ be the
			contribution where each selected $z_\gamma$- or $w_\gamma$-degree
			exceeds $X_\gamma$. Apply Cauchy's formula to the selected variables
			on circles of fixed radius $R>\max\{2,e^{c_1}\}$, leaving the other
			truncated variables at $1$. Each tail kernel is bounded by
			$R^{-X_\gamma}/(R-1)$. Changing the selected variables changes the
			diagonal, cross, and linear terms by at most
			\[
			C_R\sum_{(\gamma,\pm)\in E}
			\left(V_\gamma+|W_\gamma^+|+|W_\gamma^-|+1+
			\sum_{\delta\ne\gamma}(|V_{\gamma,\delta}|+|V_{\delta,\gamma}|)\right)
			=C_R\sum_{(\gamma,\pm)\in E}S_\gamma.
			\]
			The fixed-degree comparisons and the remainder bound are uniform on
			all these circles. Consequently
			\[
			|\mathcal Q_E|\leq C\mathcal M
			\prod_{(\gamma,\pm)\in E}
			\exp(-X_\gamma\log R+C_RS_\gamma),
			\]
			where $C$ is independent of $E$ and $H$. Unmarked variables remain
			at $1$, so we keep the cancellations in their aggregate sums.
			Choose $C_1(\log R-c_1)\geq C_R$. Since $X_\gamma\geq C_1S_\gamma$,
			inclusion--exclusion bounds the fully truncated expression by
			\[
			C\mathcal M\prod_{\gamma\in\mathfrak T}(1+e^{-c_1X_\gamma})^2
			\leq C\mathcal M\exp\left(2\sum_{\gamma\in\mathfrak T}
			e^{-c_1X_\gamma}\right)\ll\mathcal M,
			\]
			using \eqref{multiblockCutoffConditions}. This proves
			\eqref{multiblockTwistedConclusion}.
			Taking $K_p(A,B)=\mathbf1_{A=B}$ removes all linear terms. Use
			\eqref{ordinaryFixedCoefficientShift} for every fixed pair; no
			distinguished linear-term assumption is then needed. The quadratic form is
			$\sum_n|a(n)|^2/n$, so this proves \eqref{multiblockCoefficientBound}.}
	\end{proof}
	
	{\section{Rankin--Selberg prime sums}\label{appendixRankinSelberg}
		\begin{prop}\label{rankinSelbergPrimeSumsProposition}
			Let $L_1,\ldots,L_r$ be the fixed tuple from Section~1.
			With the usual omission of ramified primes, there are
			$\eta_0=\eta_0(\boldsymbol L)>0$ and constants $C_{ab}$ such that
			\begin{align}
				\sum_{p\leq x}\lambda_a(p)\overline{\lambda_b(p)}\log p
				&=\delta_{ab}x+
				O_{\boldsymbol L}\left(\frac{x}{(\log x)^{\eta_0}}\right),
				\label{rankinSelbergQuantitativePrimeTheorem}\\
				\sum_{p\leq x}\frac{\lambda_a(p)\overline{\lambda_b(p)}}p
				&=\delta_{ab}\log\log x+C_{ab}
				+O_{\boldsymbol L}\left(\frac1{(\log x)^{\eta_0}}\right)
				\label{rankinSelbergQuantitativeMertens}
			\end{align}
			uniformly for $x\geq3$.  For the prime-block endpoints of
			Section~\ref{holderReductionsSection}, put
			\[
			R_{ab,\ell}=
			\sum_{T_{\ell-1}<p\leq T_\ell}
			\frac{\lambda_a(p)\overline{\lambda_b(p)}}p
			-\delta_{ab}(\log\log T_\ell-\log\log T_{\ell-1}).
			\]
			Then
			\begin{equation}\label{rankinSelbergBlockUniformity}
				\sum_{\ell=2}^{\mathscr J}|R_{ab,\ell}|
				\ll_{\boldsymbol L,r,\widehat Y}1.
			\end{equation}
		\end{prop}
		
		\begin{proof}
			Let $\pi_a$ be the unitary automorphic representation attached to $L_a$.
			At an unramified prime, write
			\[
			L_p(s,\pi_a)=\prod_{u=1}^{d_a}
			(1-\alpha_{a,u}(p)p^{-s})^{-1}.
			\]
			Here $|\alpha_{a,u}(p)|=1$: this is elementary in degree one,
			and follows from Deligne \cite{Deligne} or Deligne--Serre
			\cite{DeligneSerre} for holomorphic forms.
			The Rankin--Selberg local factor is
			\[
			L_p(s,\pi_a\times\widetilde\pi_b)
			=\prod_{u=1}^{d_a}\prod_{v=1}^{d_b}
			(1-\alpha_{a,u}(p)\overline{\alpha_{b,v}(p)}p^{-s})^{-1}.
			\]
			Consequently $-L'/L(s,\pi_a\times\widetilde\pi_b)$ has prime coefficient
			$\lambda_a(p)\overline{\lambda_b(p)}\log p$.
			The higher prime powers contribute $O(\sqrt x\log^2x)$ to the
			corresponding summatory function.  The finitely many ramified primes
			also contribute a power-saving error.
			
			The representations in the tuple are pairwise non-isomorphic.
			Thus $L(s,\pi_a\times\widetilde\pi_b)$ has a simple pole at $1$
			exactly when $a=b$, and is otherwise holomorphic there
			\cite[Sections~5.12--5.13]{IwaniecKowalski}; it is non-vanishing on
			$\Re(s)=1$ \cite{Sarnak1}.
			Brumley's narrow zero-free region
			\cite[Corollary~6]{Brumley}, applied to this fixed tuple, gives
			a zero-free width $\gg U^{-B}$ up to height $U$.
			In the smoothed Perron argument for $-L'/L$, take
			$U=(\log x)^\xi$ with $\xi B<1$.  The shifted vertical integral then
			has a factor $\exp(-c(\log x)^{1-\xi B})$, while the truncation error
			is $O(x/U)$.
			Choose $0<\eta_0<\xi$.  Bounded prime coefficients and
			Brun--Titchmarsh on the transition intervals permit de-smoothing,
			giving \eqref{rankinSelbergQuantitativePrimeTheorem}.
			The constants are uniform in $a,b$ because the tuple is fixed.
			
			Partial summation gives \eqref{rankinSelbergQuantitativeMertens}.
			For \eqref{rankinSelbergBlockUniformity}, the first block costs $O(1)$,
			and all later endpoint errors are
			$O((\log T_2)^{-\eta_0})$.  Their sum is bounded because
			$\mathscr J(\log T_2)^{-\eta_0}\ll1$.
		\end{proof}
	}
	

\begin{thebibliography}{0}
		\bibitem{AkbaryFodden1} A. Akbary, B. Fodden, \emph{Lower bounds for power moments of {$L$}-functions}, Acta Arith. {\bf 151} no. 1 (2012), 11--38.
		
		\bibitem{AndersenThorner} {N. Andersen, J. Thorner, \emph{Zeros of $\mathrm{GL}_2$ $L$-functions on the critical line}, Forum Math. {\bf 33} no. 2 (2021), 477--491.}
		
		\bibitem{ArguinBailey1} L-P. Arguin, E. Bailey, \emph{Large deviation estimates of {S}elberg's {C}entral {L}imit
			{T}heorem and applications}, Int. Math. Res. Not. IMRN no. 23 (2023), 20574--20612. 
		
		\bibitem{ArguinBailey2} L.P. Arguin, E. Bailey, \emph{Lower bounds for the large deviations of {S}elberg's central limit theorem}, Mathematika {\bf 71} no. 1 (2025).
		
		\bibitem{BettinBuiLiRadziwill} S. Bettin, H.M. Bui, X. Li and M. Radziwi\l\l, \emph{A quadratic divisor problem and moments of the {R}iemann zeta-function}, J. Eur. Math. Soc. (JEMS) {\bf 22} no. 12 (2020), 3953--3980.
		
		\bibitem{BettinChandeeRadziwill} S. Bettin, V. Chandee, M. Radziwi\l\l, \emph{The mean square of the product of the {R}iemann zeta-function with {D}irichlet polynomials}, J. Reine Angew. Math. {\bf 729} (2017), 51--79.
		
		\bibitem{BombHej1} E. Bombieri, D.A. Hejhal, \emph{On the distribution of zeros of linear combinations of Euler products}, Duke Math. J. {\bf 80} no.3 (1995), 821--862.
		
		\bibitem{BlomerHarcos} {V. Blomer, G. Harcos, \emph{The spectral decomposition of shifted convolution sums}, Duke Math. J. {\bf 144} no. 2 (2008), 321--339.}
		
		\bibitem{BlomerHarcosErratum} {V. Blomer, G. Harcos, \emph{Twisted $L$-functions over number fields and Hilbert's eleventh problem}, Geom. Funct. Anal. {\bf 20} no. 1 (2010), 1--52; erratum dated November 1, 2013.}
		
		\bibitem{Brumley} {F. Brumley, \emph{Effective multiplicity one on $\mathrm{GL}(n)$ and narrow zero-free regions for Rankin--Selberg $L$-functions}, Amer. J. Math. {\bf 128} no. 6 (2006), 1455--1474.}
		
		\bibitem{ConreyFarmerKeatingRubinsteinSnaith} J.B. Conrey, D.W. Farmer, J.P. Keating, M.O. Rubinstein, N.C. Snaith, \emph{Integral moments of {$L$}-functions}, Proc. London Math. Soc. (3) {\bf 91} no. 1 (2005), 33--104.
		
		\bibitem{ConreyGonek1} J.B. Conrey and S.M. Gonek, \emph{High moments of the {R}iemann zeta-function}, Duke Math. J. {\bf 107} no. 3 (2001), 577--604.
		
		\bibitem{Deligne} P. Deligne, \emph{La conjecture de {W}eil. {I}}, Inst. Hautes \'Etudes Sci. Publ. Math. {\bf 43} (1974), 273--307.
		
		\bibitem{DeligneSerre} {P. Deligne and J.-P. Serre, \emph{Formes modulaires de poids $1$}, Ann. Sci. \'{E}cole Norm. Sup. (4) {\bf 7} no. 4 (1974), 507--530.}
		
		\bibitem{DurkanKarakMahatab} B. Durkan, N. Karak, K. Mahatab, \emph{Sharp lower bounds for shifted moments of Dedekind zeta functions}, preprint, \href{https://arxiv.org/abs/2609.01101}{arxiv:2609.01101}.
		
		\bibitem{Gabriel} {R. M. Gabriel, \emph{Some results concerning the integrals of moduli of regular functions along certain curves}, J. London Math. Soc. {\bf 2} (1927), 112--117.}
		
		\bibitem{GonekHughesKeating1} S.M. Gonek, C.P. Hughes, J.P. Keating, \emph{A hybrid {E}uler-{H}adamard product for the {R}iemann zeta function}, Duke Math. J. {\bf 136} no. 3 (2007), 507--549.
		
		\bibitem{GunKumarThakur} S. Gun, G. Kumar, D. Thakur, \emph{Lower {B}ounds for {M}oments of $L$-functions}, preprint, \href{https://arxiv.org/abs/2608.30061}{arxiv:2608.30061}.
		
		\bibitem{Hagen1} M.V. Hagen, Sharp conditional moment bounds for products of
		$L$-functions, Q. J. Math. {\bf 76} no. 2 (2025), 683--713.
		
		\bibitem{HardyLittlewood1} G.H. Hardy and J.E. Littlewood, \emph{Contributions to the theory of the {R}iemann zeta-function and the theory of the distribution of primes}, Acta Math. {\bf 41} no. 1 (1916), 119--196.
		
		\bibitem{Harper1} A. Harper, \emph{Sharp conditional bounds for moments of the Riemann zeta function}, preprint, \href{https://arxiv.org/abs/1305.4618}{arxiv:1305.4618}.
		
		\bibitem{Heap2} W. Heap, \emph{Moments of the {D}edekind zeta function.}
		PhD thesis, University of York (2013).
		
		\bibitem{Heap1} W. Heap, \emph{The twisted second moment of the {D}edekind zeta function of a quadratic field}, Int. J. Number Theory {\bf 10} no. 1 (2014), 235--281.
		
		\bibitem{Heap3} W. Heap, \emph{On the splitting conjecture in the hybrid model for the Riemann zeta function}, Forum Mathematicum {\bf 35}, no.2 (2023), 329--362.
		
		\bibitem{HeapLi} W. Heap and J. Li, \emph{Simultaneous extreme values of zeta and {$L$}-functions}, Math. Ann. {\bf 390} no. 4 (2024), 6355--6397.
		
		\bibitem{HeapRadSound1} W. Heap, M. Radziwi\l\l\, and K. Soundararajan, \emph{Sharp upper bounds for fractional moments of the {R}iemann zeta function}, Q. J. Math. {\bf 70} no. 4 (2019), 1387--1396.
		
		\bibitem{HeapSound1} W. Heap and K. Soundararajan, \emph{Lower bounds for moments of zeta and {$L$}-functions revisited}, Mathematika {\bf 68} no.1 (2022), 1--14.
		
		\bibitem{Heath-Brown1}{D. R. Heath-Brown, \emph{Fractional moments of the {R}iemann zeta-function}, J. London Math. Soc. (2) {\bf 24} no. 1 (1981), 65--78.}
		
		\bibitem{Ingham1} A.E. Ingham, \emph{Mean-{V}alue {T}heorems in the {T}heory of the {R}iemann {Z}eta-{F}unction}, Proc. London Math. Soc. (2) {\bf 27} no. 4 (1927), 273--300.
		
		\bibitem{Inoue1} S. Inoue, \emph{On the logarithm of the {R}iemann zeta function and its
			iterated integrals}, Ann. Sc. Norm. Super. Pisa Cl. Sci. (5) {\bf 27} no. 1 (2026), 491--533.
		
		\bibitem{InoueLi1} S. Inoue and J. Li, \emph{Joint value distribution of {$L$}-functions on the critical line}, preprint, \href{https://arxiv.org/abs/2102.12724}{arxiv:2102.12724}.
		
		\bibitem{IwaniecKowalski} H. Iwaniec and E. Kowalski, \emph{Analytic number theory}, American Mathematical Society Colloquium Publications {\bf 53}, American Mathematical Society, Providence, RI (2004).
		
		\bibitem{KeatingSnaith1} J.P. Keating and N.C. Snaith, \emph{Random matrix theory and {$\zeta(1/2+it)$}}, Comm. Math. Phys. {\bf 214} no. 1 (2000), 57--89.
		
		\bibitem{KimSarnak} H. Kim and P. Sarnak, \emph{Appendix: Refined estimates towards the Ramanujan and Selberg conjectures}, J. Amer. Math. Soc. {\bf 16} no. 1 (2003), 175--181.
		
		\bibitem{MilinoButter1} M. Milinovich, C. Turnage-Butterbaugh, \emph{Moments of the product of automorphic $L$-functions}, J. Number Theory {\bf 139} (2014), 175--204.
		
		\bibitem{Motohashi1} Y. Motohashi, \emph{A note on the mean value of the {D}edekind zeta-function of the quadratic field}, Math. Ann. {\bf 188} (1970), 123--127.
		
		\bibitem{Rad1} M. Radziwi\l\l, \emph{Large deviations in Selberg’s central limit theorem}, preprint, \href{https://arxiv.org/abs/1108.5092}{arxiv:1108.5092}. 
		
		\bibitem{RS1} M. Radziwi\l\l, K. Soundararajan, \emph{Continuous lower bounds for moments of zeta and {$L$}-functions}, Mathematika {\bf 59} no. 1 (2013), 119--128.	
		
		\bibitem{RS2} M. Radziwi\l\l, K. Soundararajan, \emph{Moments and distribution of central $L$-values of quadratic twists of elliptic curves}, Invent. Math. {\bf 202} no. 3 (2015), 1029--1068.
		
		\bibitem{Rane1} V. V. Rane, \emph{On an approximate functional equation for {D}irichlet {$L$}-series}, Math. Ann. {\bf 264} no. 2 (1983), 137--145.
		
		\bibitem{Sarnak1} P. Sarnak, \emph{Nonvanishing of {$L$}-functions on {$\mathfrak{R}(s)=1$}}, Contributions to automorphic forms, geometry, and number
		theory, Johns Hopkins Univ. Press, Baltimore, MD (2004), 719--732.
		
		\bibitem{Sel1} A. Selberg, \emph{Old and new conjectures and results about a class of {D}irichlet series}, Collected Papers, Vol. II, Springer-Verlag, Berlin, 1991, 47--63.
		
		\bibitem{Sono1} K. Sono, \emph{Continuous lower bounds for the moments of {D}edekind zeta-functions}, J. Number Theory {\bf 188} (2018), 335--356.
		
		\bibitem{Sound1} K. Soundararajan, \emph{Moments of the Riemann zeta function}, Annals of Math. {\bf 170} no. 2 (2009), 981--993. 
		
		\bibitem{Sourmelidis1} A. Sourmelidis, \emph{Joint extreme values of {$L$}-functions on and off the critical line}, preprint, \href{https://arxiv.org/abs/2605.03665}{arxiv:2605.03665}.
		
		\bibitem{Titchmarsh} E. C. Titchmarsh, \emph{The theory of the {R}iemann zeta-function}, 2nd ed., The Clarendon Press, Oxford University Press, New York, 1986.
		
		
	\end{thebibliography}
\end{document}